\documentclass[]{article}
\usepackage{graphicx}
\usepackage{caption}
\usepackage{amsfonts}
\usepackage{amsmath}
\usepackage{amssymb}
\usepackage{fancyhdr}
\usepackage{titlesec}
\usepackage{indentfirst}
\usepackage{booktabs}
\usepackage{verbatim}
\usepackage{color}
\usepackage{tcolorbox}
\usepackage{mathtools}
\usepackage{bm}
\usepackage{amsthm}
\usepackage{wasysym}
\usepackage{empheq}
\graphicspath{{figures/}} 
\usepackage[font=small,labelfont=bf]{caption} 
\usepackage{chngcntr}

\newcommand{\rom}[1]{\uppercase\expandafter{\romannumeral #1\relax}}
\usepackage[page,header]{appendix}
\usepackage{subfig}

\usepackage{titletoc}

\numberwithin{equation}{section}
\newtheorem{theorem}{Theorem}[section]
\newtheorem{lemma}[theorem]{Lemma}
\newtheorem{corollary}[theorem]{Corollary}
\newtheorem{proposition}[theorem]{Proposition}

\newtheorem{remark}[theorem]{Remark}
\newtheorem{condition}[theorem]{Assumption}

\def\thetheorem {{\arabic{section}.\arabic{theorem}}}

\usepackage{tikz}
\newcommand*\circled[1]{\tikz[baseline=(char.base)]{
            \node[shape=circle,draw,inner sep=0.8pt] (char) {#1};}}

\def\k{\kappa}

\def\d{\delta}

\def\g{\gamma}

\def\o{\omega}
\def\p{\phi}

\def\D{\Delta}
\def\p{\t{\rho}}
\def\Up{\Upsilon}

\def\bz{{\bf{z}}}
\def\ba{{\bm{\alpha}}}

\def\R{\mathbb R}

\def\l{\left}
\def\r{\right}
\def\la{\left\langle}
\def\ra{\right\rangle}
\def\ll{\left\lVert}
\def\rl{\right\rVert}
\def\lv{\left\lvert}
\def\rv{\right\rvert}
\def\({\left(}
\def\){\right)}
\def\[{\left[}
\def\]{\right]}

\def\pt{\partial}

\def\qd{\quad}

\def\t{\tilde}

\def\h{\hat}
\def\bn{\binom}

\def\m{\t{m}}
\def\pinf{\rho^\infty}
\def\minf{\m^\infty}
\def\rinf{r^\infty}
\def\tp{\rho}
\def\tm{m}
\def\tt{{{\theta}}}
\def\Tpn{{{\bm{\tp}}}^n}
\def\Tpl{{{\bm{\tp}}}^l}
\def\Tmn{{{\bm{\tm}}}^n}
\def\Tml{{{\bm{\tm}}}^l}

\def\Ttn{{{\bm{\theta}}}^n}

\def\bPk{{\bf{\Phi}}^K}
\def\hp{\h{\tp}}
\def\hpk{\h{\tp}^K}
\def\Hpk{\h{{\bm{\tp}}}^K}
\def\hm{\h{\tm}}
\def\hmk{\h{\tm}^K}
\def\Hmk{\h{{\bf{\tm}}}^K}
\def\Htk{\h{{\bm{\theta}}}^K}
\def\hht{\h{\theta}}
\def\htk{\h{\theta}^K}

\def\bpk{\bar{\tp}^K}
\def\bmk{\bar{\tm}^K}
\def\epk{\varepsilon_\tp^K}
\def\ppk{\varrho_\tp^K}
\def\Bpk{{\bm{\varepsilon}}_\tp^K}
\def\emk{\varepsilon_\tm^K}
\def\pmk{\varrho_\tm^K}
\def\Bmk{{\bm{\varepsilon}}_\tm^K}

\def\Ek{\h{E}^K}
\def\Eek{E_\varepsilon^K}

\def\Epi{E^n_\pi}

\def\u{\mu}
\def\ta{\eta}

\def\EHn{E_{H^n_\pi}}

\begin{document}

\title{An Uncertainty Quantification Approach to the Study of Gene Expression Robustness}

\author{ 
Pierre Degond\footnote{Department of Mathematics, Imperial College London, South Kensington Campus,
London SW7 2AZ, United Kingdom. (pdegond@imperial.ac.uk)}\ \ \ \ \
Shi Jin\footnote{School of Mathematical Sciences, Institute of Natural Sciences, MOE-LSEC and SHL-MAC, Shanghai Jiao Tong University, Shanghai 200240, China. (shijin-m@sjtu.edu.cn) }\ \ \ \ \  Yuhua Zhu\footnote{Department of Mathematics, Stanford University, Stanford, CA 94305, United States. (yuhuazhu@stanford.edu)}  \ \  \  \ \  }

\date{}
\maketitle

\vspace{-0.7cm}
\begin{center}
Dedicated to Professor Ling Hsaio’s 80th birthday
\end{center}

\vspace{0.2cm}

\begin{abstract}
We study { a chemical kinetic system} with uncertainty { modeling a gene regulatory network in biology. Specifically, we consider a system of two equations for the messenger RNA and micro RNA content of a cell. Our target is to provide a simple framework for noise buffering in gene expression through micro RNA production.} Here the uncertainty, modeled by random variables, enters the system through the initial data and the source term.  We obtain a sharp decay rate of the solution to the steady state, which reveals that the biology system is not sensitive to the initial perturbation around the steady state.  The sharp regularity estimate leads to the stability of the generalized Polynomial Chaos stochastic Galerkin (gPC-SG) method.  Based on the smoothness of the solution in the random space and the stability of the numerical method, we conclude the gPC-SG method has spectral accuracy.  Numerical experiments are conducted to verify the theoretical findings. 
\end{abstract}

{\small


}

\section{Introduction}
\label{sec:Intro}
{
In this paper, we are interested in a model of a simple gene regulatory network describing the regulation of the transcription of nuclear DNA by microRNAs (further denoted by $\mu$RNA). The synthesis of a protein from its DNA sequence involves several steps: the binding of a transcription factor (which can be a protein or another type of molecule) on the gene promotor sequence initiates the transcription of DNA into messenger RNA (further denoted by mRNA). mRNA is later translated into proteins in the ribosomes. Here, we are specifically interested in the first step, i.e. the transcription of DNA into mRNA. This transcription is subject to a high level of noise due for instance to noise in the availability of transcription factors. Yet, cells have to perform functions with a high level of precision and some noise buffering systems must be at play. In recent years, the role of $\mu$RNA has attracted focus. $\mu$RNAs are very short RNA sequences which are coded by non protein-coding sequences of the nuclear DNA. They seem to have (among other roles) a role in the regulation of transcription. Indeed, in many cases, the transcription factor initiates transcription of DNA into both the mRNA and a regulatory $\mu$RNA. The synthetized $\mu$RNA binds on the mRNA and prevents its translation into proteins. It has been argued that the main function of this regulation is to reduce the effect of noise in the transcription process (see \cite{Bleris_MolSystBiol11, Blevins_PlosGenetics15} and the review \cite{Herranz_GenesDev10}).

Our model involves a pair of chemical kinetic equations for the mRNA and $\mu$RNA content, with source terms modeling the action of the transcription factor. The effect of the noise is taken into account by adding some uncertainty in the noise term and the initial data. We are interested in looking at how this uncertainty propagates to the mRNA content and in comparing this uncertainty between situations including $\mu$RNA production or not. The uncertainty is modeled by random variables with given probability density functions. 

A classical approach to the study of noise in gene regulatory networks is through the chemical master equation \cite{vanKampen} which is solved numerically by means of Gillespie's algorithm \cite{Gillespie_JPhysChem77}, see e.g. \cite{Bosia_BMCSystBiol12, Osella_PlosCB11}. Here, we use the chemical kinetic approach, which is a valid approximation of the chemical master equation when ther number of molecules is large. However, this approximation is far from being valid in a cell. This is why we mitigate this discrepancy by assuming a random availability of transcription factors. The advantage is a considerably simpler treatment than with the chemical master equation while preserving the important features of the system. An alternate approach presented in \cite{Deg_Her_Mir} considers Brownian perturbations in the chemical kinetic equations. Introducing the joint probability density for mRNA and $\mu$RNA leads to a Fokker-Planck equation which can be analytically solved under some time-scale separation hypotheses. Underlying this approach is the idea that random perturbations do not only affect the initial condition and the source term, but are present at all times. In the present work, we restrict to random perturbation of the source term and initial data which allows us to use the simpler framework of uncertainty quantification. 
}

We will mainly focus on two aspects of this problem. First, we study how a random perturbation near the steady state will affect the system by analyzing the long-time behavior of the perturbative solution in the random space in terms of the weighted Sobolev norm $H^n_\pi$, where $\pi$ is the probability density function of the random variable. We also study the stability and the convergence rate of a numerical method to the system with uncertainty, specifically, the generalized Polynomial Chaos approximation based stochastic Galerkin (gPC-SG) method.

There are plenty of developments regarding the sensitivity analysis and convergence analysis in uncertainty quantification.  For example, the solution to elliptic equations,  parabolic equations, \cite{babuska2004galerkin, cohen2010convergence, cohen2011analytic}, and kinetic equations \cite{HuJin18, li2016uniform, ZhuJin18, jin2016august,  Jintwophase, JinLiu18, Zhu2017BE, ZhuVPFP18}. To our knowledge, there has been no such analysis done { to a system of chemical kinetic equations describing a gene regulatory network}.

There are mainly two difficulties in the analysis. The first one is in the sensitivity analysis in the random space. When we do estimates on the Sobolev norm $H^n_\pi$, the size of the nonlinear terms will increase to $O(2^n)$. This will result in a strong assumption on the initial data, that is, the initial randomness is required to be as small as $O(1/2^n)$ to get an exponential decay. Similarly, when we do the stability analysis of the gPC-SG method, if we approximate the solution by $K$-th order polynomial chaos bases, the size of the nonlinear terms in the resulting deterministic system will be $O(K!)$. If we directly do energy estimates on the approximate solution, then we can only prove stability when the initial randomness is as small as $O(1/K!)$. To sum up, how to get a sharp estimate in terms of $n$ and $K$ without strong assumption on the initial data or steady state is the main difficulty in this problem.

 In this paper, we obtain a sharp decay of the random perturbation around the steady state in terms of its Sobolev norm $H^n_\pi$ through a carefully designed weighted energy norm.  Under some mild conditions on the initial data that is independent of $n$, we find that the random perturbation near the steady state will decay exponentially. Our results also reveal that the solution preserves the regularity in the random space.  Moreover, with another weighted energy norm, we prove the stability of the $K$-th order gPC-SG method with an assumption on initial data independent of $K$. The smoothness of the solution in the random space and the stability of the gPC-SG method allows us to prove the spectral convergence of the gPC-SG method. When approximating the numerical solution by the $K$-th order polynomial chaos basis, the error of the approximation solution in $H^n_\pi$ is $O(K^{-n})$.

This paper is organized as follows. Section \ref{sec: model} gives an introduction { to the chemical kinetic system modeling the targeted gene regulatory network} and its corresponding steady state. The main result and proof sketch about the sensitivity of the system under random perturbation near steady state is stated in Section \ref{sec: sensitivity results}. The proof of this result is in the following Section \ref{sec: sensitivity}.  In Section \ref{sec: gPC method}, the gPC-SG method is introduced and the stability and the convergence rate of this method are stated in Section \ref{sec: gPC result}. The proof of these two results are in Section \ref{sec: stability} and \ref{sec: spectral} respectively. In Section \ref{sec: numerics} we numerically study how the presence of RNA inuences the noise in the concentration of
unbound mRNA.

\section{The model}
\label{sec: model}
Consider the following model,
\begin{equation}
\label{eq: model}
\begin{cases}
&\pt_t\p = S(z) - a\p - c\p \m,\\
&\pt_t \m = S(z) - b\m -c\p\m,
\end{cases}
\end{equation}
with initial data $\p(0,z), \m(0,z).$ $a, b, c$ are positive constants. Here $\p(t,z), \m(t,z)$ { respectively stand for the content of unbound mRNA and $\mu$RNA of a cell at time $t$. $S(z)$ is the source term which models the production of mRNA and $\mu$RNA through DNA transcription. We assume that a molecule of $\mu$RNA is produced each time a molecule of mRNA is produced, hence the same source term arises in the two equations. The production of mRNA and $\mu$RNA is subject to the availability of the transcription factor, which is random. We model this randomness by assuming that the source term is a given function of a random variable $z$ (modelling for instance the concentration of transcription factors) with probability density function $\pi(z)$ on a compact set $I_z \subset \R$. The first equation of \eqref{eq: model} models the decay of unbound mRNA through its binding to an unbound $\mu$RNA (the term $- c\p \m$) or through other degradation mechanisms (the term $- a\p$). The second equation of \eqref{eq: model} describes the decay of unbound $\mu$RNA through its binding to an unbound mRNA (the term $- c\p \m$ again) and through other degradation mechanisms (the term $- b\m$). Note that the binding of $\mu$RNA to mRNA consumes one molecule of $\mu$RNA and one molecule of mRNA at the same time, which explains why the same loss term is involved in the two equations. The remaining unbound mRNA is then supposed to enter the translation process into proteins through the actions of ribosomes. This step is supposed to occur later and is not included in the model. }

If one sets $\pt_t \p = \pt_t\m = 0$, one can get the steady state $\pinf(z), \minf(z)$, 
\begin{equation}
\label{eq: stationary}
 \pinf = b\rinf,\qd \minf = a\rinf,\qd \text{with } \rinf(z) =\frac{1}{2c}\(-1+\sqrt{\D}\)>0, \qd \D =1+\frac{4cS(z)}{ab} >1.
\end{equation}
Let $(\tp, \tm) = \(\p - \pinf,  \m - \minf\)$ be the random perturbative solution around the steady state, then $\(\tp(t,z), \tm(t,z)\)$ satisfies
\begin{equation}
\label{eq: perturb}
\l\{
\begin{aligned}
&\pt_t\tp = - \(a+ac\rinf\)\tp - bc\rinf \tm - c\tp\tm,\\
&\pt_t \tm = - \(b +bc\rinf\) \tm -ac\rinf\tp -c\tp\tm,
\end{aligned}
\r.
\end{equation}
with initial data,
\begin{equation*}
\tp(0,z) = \p(0,z) - \pinf(z), \qd \tm(0,z) = \m(0,z) - \minf(z).
\end{equation*}

\subsection{Main results and proof sketch}
\label{sec: sensitivity results}
We are interested in the estimates for the solution $\(\tp, \tm\)$ in the random space using the norm 
\begin{equation}
\ll \tp(t) \rl_{H^n_\pi}^2 = \sum_{i=0}^n\int  (\pt^i_z\tp)^2\pi(z)dz, \qd \ll \tm(t) \rl_{H^n_\pi}^2 = \sum_{i=0}^n\int (\pt^i_z\tm)^2\pi(z)dz.
\end{equation}
There are two reasons why we are interested in this Sobolev norm. First, by studying this norm, we can understand how sensitive the system with respect to the random perturbation around the steady state is and how this perturbation evolves in time. Second, this norm gives the Sobolev regularity of the solution in the random space. We will approximate the solution by the gPC-SG method in the random space in Section $\ref{sec: gPC method}$. Such regularity allows us to prove the spectral convergence of the method.

The difficulty in the analysis is to get an estimate of $\ll \tp \rl^2_{H^n_\pi}, \ll \tm \rl^2_{H^n_\pi}$ that is sharp for large $n$. For $n=0$, one can do standard energy estimates on $a\ll \tp \rl^2+ b\ll \tm \rl^2$ to get an exponential decay of the random perturbation in time under a smallness assumption on initial data. We will show the result for  $n=0$ in the following lemma, and explain why it is not trivial to extend it to   $n>0$ after the proof of the lemma. 

\begin{lemma}
\label{lemma: est_0}
    If initially, the random perturbations satisfy
    \begin{equation*}
        \ll \tp(0) \rl^2_\pi \leq \frac{b^2}{4c^2} , \qd \ll \tm(0) \rl^2_\pi \leq \frac{a^2}{4c^2},
    \end{equation*}
    then the perturbations $(\ll \tp(t) \rl^2_\pi, \ll \tm(t) \rl^2_\pi)$  decay exponentially in time as follows,
    \begin{equation*}
        \ll \tp(t) \rl^2_\pi \leq \frac{1}{a}\(a\ll \tp(0) \rl^2_\pi + b\ll \tm(0) \rl^2_\pi\)e^{-at},  \qd \ll \tm(t) \rl^2_\pi \leq \frac{1}{b}\(a\ll \tp(0) \rl^2_\pi + b\ll \tm(0) \rl^2_\pi\)e^{-bt}.
    \end{equation*}
\end{lemma}

\begin{proof}
Multiplying $a\tp$ and $b\tm$ to the two equations in (\ref{eq: perturb}) respectively, and then adding them together gives,
\begin{equation}
\label{eq: n=0}
\begin{aligned}
    \frac{1}{2}\pt_t\(a\tp^2 + b\tm^2\) =& - a^2\tp^2 - b^2\tm^2\underbrace{-\(a^2c\rinf \tp^2 + 2abc\rinf \tm\tp + b^2c\rinf\tm^2\)}_{\text{linear part}} \underbrace{- ac\tp^2\tm - bc\tm^2\tp}_{\text{nonlinear part}}\\
    \leq & -a^2\tp^2 -b^2\tm^2\underbrace{ -c\rinf \(a\tp + b\tm \)^2 }_{\text{linear part}}\underbrace{+ c^2\tp^2\tm^2 + \frac{a^2}{4}\tp^2 + c^2\tp^2\tm^2 + \frac{b^2}{4}\tm^2}_{\text{nonlinear part}} \\
    \leq & - \(\frac{3a}{4} - \frac{c^2}{a}\tm^2 \)a\tp^2 - \(\frac{3b}{4}  - \frac{c^2}{b}\tp^2 \)b\tm^2,
\end{aligned}
\end{equation}
where we apply Young's inequality to the nonlinear part in the first line to obtain the first inequality. Since $\rinf$ defined in (\ref{eq: stationary}) is always positive for any $z\in I_z$, this gives $-c\rinf \(a\tp + b\tm \)^2\leq 0$, so we can omit this term in the second inequality. 

After we obtain the inequality as in (\ref{eq: n=0}), the exponential decay of $\tp^2, \tm^2$ follows from a smallness assumption on the initial condition.  Assume the coefficients of $a\tp^2$ and $b\tm^2$ on the RHS of (\ref{eq: n=0}) are smaller than $-\frac{a}{2}$ and  $-\frac{b}{2}$ respectively, which is equivalent to assume
\begin{equation}
\label{eq: est_0 initial with z}
    \tm^2(0,z) \leq \frac{a^2}{4c^2} , \qd \tp^2(0,z) \leq \frac{b^2}{4c^2}, \qd \text{for all } z\in I_z,  
\end{equation}
then by continuity argument, for all $t>0$, one has
\begin{equation*}
    \frac{1}{2}\pt_t\(a\tp^2 + b\tm^2\) \leq  - \frac{a^2}{2}\tp^2 - \frac{b^2}{2} \tm^2.
\end{equation*}
Integrating the above equation over time, one gets, 
\begin{equation*}
    a\tp^2(t) + b\tm^2(t) \leq  a\tp^2(0) + b\tm^2(0) - \int_0^t a^2\tp(s)^2ds -\int_0^t b^2 \tm(s)^2ds, 
\end{equation*}
which implies
\begin{equation*}
\begin{aligned}
    &a\tp^2(t) \leq  a\tp^2(0) + b\tm^2(0) - a^2\int_0^t\tp(s)^2ds ,\\
    & b\tm^2(t) \leq  a\tp^2(0) + b\tm^2(0) - b^2\int_0^t \tm(s)^2ds.
\end{aligned}     
\end{equation*}
By Grownwall's inequality, one can get the exponential decay of $\tp^2, \tm^2$ as follows, 
\begin{equation*}
    \tp^2(t,z) \leq  \frac{1}{a}\(a\tp^2(0,z) + b\tm^2(0,z)\)e^{-at},\qd \tm^2(t,z) \leq  \frac{1}{b}\(a\tp^2(0,z) + b\tm^2(0,z)\)e^{-bt}.     
\end{equation*}

Finally, one integrates (\ref{eq: est_0 initial with z}) and the above estimates over $\pi(z)dz$, one completes the proof of Lemma \ref{lemma: est_0}.

\end{proof}

The difficulties of extending the results in Lemma \ref{lemma: est_0} to $\ll \tp \rl^2_{H^n_\pi}, \ll \tm \rl^2_{H^n_\pi}$ are mainly due to two reasons. First when $n=0$, the linear part in (\ref{eq: n=0}) is a negative square without any assumption on $\rinf$, so we can directly omit these terms in the estimates. However, if we directly do energy estimates on $\ll \tp \rl^2_{H^n_\pi}, \ll \tm \rl^2_{H^n_\pi}$ for $n>0$, we have to assume $\sum_{i=1}^n\lv \pt_z^i\rinf \rv \leq O(1/n!)$ to make the linear part negative and this assumption is too strong. Second, the nonlinear part will be $O(n(n!)^2)$  if we directly estimate on $\ll \tp \rl^2_{H^n_\pi}, \ll \tm \rl^2_{H^n_\pi}$ for $n>0$. This implies that one needs to assume the initial data $\ll \tp(0) \rl^2_{H^n_\pi}, \ll \tm(0) \rl^2_{H^n_\pi}$ as small as $O(1/n/(n!)^2)$ to get the exponential decay in time. We will explain it in more details in the following paragraph. 

In order to simplify the notation, we set 
\begin{equation}
\label{def of Tpn}
\begin{aligned}
    &\tt = a\tp+b\tm;\qd \\
    &\Tpn = (\tp, \pt_z\tp, \cdots, \pt_z^n\tp), \qd \text{similar for } \Tmn, \Ttn,
\end{aligned}
\end{equation}
and let $\ll \cdot \rl^2$ be the regular  Euclidean norm, 
\begin{equation*}
    \ll \Tpn \rl^2 = \sum_{l=0}^n\tp^2_l.
\end{equation*}
If we directly do energy estimates on $a\ll \Tpn \rl^2 +b\ll \Tmn \rl^2$, then we will get the following inequality
by taking $\pt_z^l$ ($0\leq l\leq n$) to (\ref{eq: perturb}), then multiplying $a\tp_l, b\tm_l$ respectively and adding all equations together, 
\begin{equation}
\label{eq: pt_zl perturb}
\begin{aligned}
&\frac{1}{2}\pt_t\(a\ll \Tpn \rl^2 + b\ll \Tmn \rl^2\) \\
= &-a^2\ll\Tpn\rl^2-b^2\ll\Tmn\rl^2 + \underbrace{\sum_{l=0}^n\l[- c\pt_z^l\(\rinf \tt \) \r]\tt_l}_{\text{linear part}} +\underbrace{\sum_{l=0}^n- ac\pt_z^l(\tp\tm)\pt_z^l\tp- bc\pt_z^l(\tp\tm)\pt_z^l\tm}_{\text{nonlinear part}}
\end{aligned}
\end{equation} 
First, for the linear part, when $n=0$, the linear part is automatically a negative square term, so we do not need to bound this term any more. However, when $n>0$, since $\rinf$ in the linear term depends on $z$, so taking $\pt_z^l$ to the linear terms gives
\begin{equation}
\label{eq: explain linear proc}
\begin{aligned}
    &\text{linear part}
    = -c\sum_{l=0}^n\pt_z^l\(\rinf\tt\)\tt_l = -c\rinf\ll \Ttn \rl^2-c\sum_{l=1}^n\sum_{i=1}^l\bn{l}{i}\pt_z^i\rinf\(\pt_z^{l-i}\tt\)\(\pt_z^l\tt\).
\end{aligned}
\end{equation}
Since $-\pt_z^i\rinf(z), i\geq1$ are not necessarily negative, only the first term in the last equality of the above equation is negative. Therefore,  we need to bound all other terms using the first negative term. By applying Young's inequality and Cauchy-Schwatz inequality to all other terms gives, 
\begin{equation*}
\begin{aligned}
\lv-c\sum_{l=1}^n\sum_{i=1}^l\bn{l}{i}\pt_z^i\rinf\(\pt_z^{l-i}\tt\)\(\pt_z^l\tt\) \rv \leq& \frac{c}{2}\sum_{l=1}^n\sum_{i=1}^l\bn{l}{i}\lv\pt_z^i\rinf\rv\(\(\pt_z^{l-i}\tt\)^2 + \(\pt_z^l\tt\)^2\)\\
\lesssim& \bn{n}{[n/2]}\(\sum_{i=1}^n\lv\pt_z^i\rinf\rv\) \ll \Ttn \rl^2,
\end{aligned}
\end{equation*}
where $[n/2]$ represents the smallest integer that is larger than or equal to $n/2$. The coefficient can be upper bounded by
\begin{equation*}
    \bn{n}{[n/2]} \leq 2^n,
\end{equation*}
this  implies only when 
\begin{equation}
\label{eq: explain linear proc_2}
    \sum_{i=1}^n\lv\rinf_i\rv \leq O\(\frac{\rinf}{2^n}\),
\end{equation}
the RHS of  (\ref{eq: explain linear proc}) is non-positive. Obviously, the constraint (\ref{eq: explain linear proc_2}) on $\rinf(z)$ is too strong. Only a small set of steady states are included in this analysis.  So we will develop another method to avoid that.

Second, for the nonlinear part in (\ref{eq: pt_zl perturb}), since the two terms are similar, we only estimate the first nonlinear term. Applying Young's inequality gives
\begin{equation}
\label{eq: pm explain}
\begin{aligned}
    &\lv\sum_{l=0}^n ac\pt_z^l\(\tp\tm\) \pt_z^l\tp \rv \leq c^2\sum_{l=0}^n \(\pt_z^l\(\tp\tm\)\)^2 + \frac{a^2}{4}\ll \Tpn \rl^2=\sum_{l=0}^n\(\sum_{i=0}^l\bn{l}{i}\pt_z^i\tm\pt_z^{l-i}\tp\)^2 + \frac{a^2}{4}\ll \Tpn \rl^2\\
    \leq& \sum_{l=0}^n\(\sum_{i=0}^l\bn{l}{i}\(\pt_z^i\tm\)^2\)\(\sum_{i=0}^l\bn{l}{i}\(\pt_z^{l-i}\tp\)^2\)+ \frac{a^2}{4}\ll \Tpn \rl^2
     \leq \bn{l}{[l/2]}^2\ll \Tml \rl^2 \ll\Tpl\rl^2+ \frac{a^2}{4}\ll \Tpn \rl^2\\
     \leq&  c^2 n\bn{n}{[n/2]}^2\ll \Tmn \rl^2 \ll \Tpn \rl^2+ \frac{a^2}{4}\ll \Tpn \rl^2 \leq c^2 n 2^{2n}\ll \Tmn \rl^2 \ll \Tpn \rl^2+ \frac{a^2}{4}\ll \Tpn \rl^2,
\end{aligned}
\end{equation}
 where the first inequality comes from Cauchy-Schwatz inequality. 
One can get similar inequality for $\lv\sum_{l=0}^n bc\pt_z^l(\tp\tm)\pt_z^l\tm\rv$. Therefore, if one ignores the linear terms in (\ref{eq: pt_zl perturb}), one ends up with the following estimates,
\begin{equation*}
\begin{aligned}
    &\frac{1}{2}\pt_t\(a\ll \Tpn \rl^2+b \ll \Tmn \rl^2\) \\
    \leq & -\(\frac{3a}{4} - \frac{c^2}{a} 2^{2n}n \ll \Tmn \rl^2\)a\ll \Tpn \rl^2-\(\frac{3b}{4} - \frac{c^2}{b} 2^{2n}n \ll \Tpn \rl^2\)b\ll \Tmn \rl^2, 
\end{aligned}
\end{equation*}
which implies that we have to assume 
\begin{equation}
\label{eq: explain nonlinear proc_2}
    \ll \Tpn(0) \rl^2,  \ll \Tmn(0) \rl^2 \leq O\(\frac{1}{  2^{2n} n}\)
\end{equation}
to get an exponential decay as follows
\begin{equation*}
    \ll \Tpn(t) \rl^2 \leq O(e^{-at}), \qd \ll \Tmn(t) \rl^2 \leq O(e^{-bt}).
\end{equation*}
If one integrates the above two equations over $\pi(z)dz$, then one will get the corresponding result in the Sobolev space. However, this result is too weak for large $n$. If the initial perturbation is smooth enough in the random space, then $\ll \tp(0) \rl , \ll \tm(0) \rl \in H^n_\pi$ for any large n.  However, by the above result, only for the initial random perturbation $\ll \tp(0) \rl^2_{H^n_z}, \ll \tm(0) \rl^2_{H^n_z}$ that are as small as $O(1/4^nn)$,  then $\ll \tp(t) \rl^2_{H^n_z}, \ll \tm(t) \rl^2_{H^n_z}$ will decay exponentially in time.

In our analysis, we overcome the two difficulties mentioned above by adding a weight $\o_i^*$ to $\tp_i, \tm_i$. Then we will only have an assumption on the initial data that is independent of $n$, furthermore, we only require $\rinf$ to satisfy the following assumption. 
\begin{condition}
\label{cond: rinf}
    There exists a constant $\k$ such that, the derivative of $\rinf$ in the random space can be bounded by
    \begin{equation}
        \label{def of k}
        \sup_{z\in I_z} |(i+1)^2\pt_z^i \rinf| \leq \k^{i+1}i!,
    \end{equation}
    and it is bounded below and above by $r, R$ respectively, 
    \begin{align}
     &r \leq \rinf \leq R, \qd \forall z\in I_z.\label{def of r, R}
\end{align}
\end{condition}
This condition is not strict at all. Actually for any analytic function $\rinf(z)$ in a compact set $I_z$, there exists a constant $C$, such that 
    \begin{equation*}
        \lv \pt_z^i \rinf\rv \leq C^{i+1}i!, \qd\text{for }\forall i \geq 0,\qd \forall z\in I_z.
    \end{equation*}
    Then set 
     \begin{equation*}
        \k = eC, 
    \end{equation*}
    one can always get 
        \begin{equation*}
        \lv (i+1)^2\pt_z^i \rinf\rv \leq \k^{i+1}i!, \qd\text{for }\forall i \geq 0,\qd \forall z\in I_z.
    \end{equation*}

The weight $\o^*_i$ we add to $\tp_i, \tm_i$ in $\ll \Tpn \rl^2, \ll \Tmn \rl^2$ is
\begin{equation*}
    \o^*_i = \frac{L^{n-i}}{\k^i}\frac{(i+1)^2}{i!}, 
\end{equation*}
where $\k$ is the constant in (\ref{def of k}), L is a constant depending on $\k$, which we will define later. In this weight, the term $\frac{(i+1)^2}{i!}$ is used to avoid strong assumption on initial data like (\ref{eq: explain nonlinear proc_2}). Notice that with $\frac{1}{i!}$, the factorial in (\ref{eq: pm explain}) can be absorbed into the weights, so one can get rid of $O(1/(n!)^2)$ in the initial assumption; while the weight $(i+1)^2$ is used to deal with $O(1/n)$ in the assumption.  Another part of the weight $\frac{L^{n-i}}{\k^i}$ is used to avoid strong constraint on $\rinf$ like  (\ref{eq: explain linear proc_2}). Under Assumption \ref{cond: rinf}, the term $\frac{1}{k^i}$ can be used to bound $\lv \rinf_i \rv$. One further notices that when $i$ is smaller, $\tt_i$ will be summed for more times, so the term $L^{n-i}$ is used to balance this. Please refer to Lemma \ref{lemma: sensitivity nonlinear} for details.

The $\frac{(i+1)^2}{i!}$ part of the weight is first introduced in \cite{Jintwophase}; However, the assumption on $\rinf$ and its corresponding weight $L^{n-i}/\k^i$ haven't been developed before.

Before we present the main theorems on the sensitivity of the perturbative solution $(\tp, \tm)$, we first list the frequently used notations here. 
\begin{itemize}
\item [--] $A, L$ are constants defined as,
\begin{align}
    &A = \sum_{i = 1}^\infty \frac{1}{i^2} = \frac{\pi^2}{6},\label{def of A}\\
    &L = \sqrt{\frac{16A\k^2}{r^2}+1}, \label{def of L}
\end{align}
where $\k$ is defined in (\ref{def of k}).
\end{itemize}

The following Theorem is about the sensitivity of the perturbative solution $(\tp, \tm)$ in the random space. 
\begin{theorem}
\label{thm: sensitivity}
For $\forall n\geq 0$, under assumption \ref{cond: rinf}, in addition, if initially 
\begin{equation}
\label{sens cond}
\begin{aligned}
    &\ll \tm(0) \rl^2_{H^n_\pi} \leq a^2C_0, \qd \ll \tp(0) \rl^2_{H^n_\pi}\leq b^2C_0, 
\end{aligned}     
\end{equation}
then the perturbative solution to (\ref{eq: perturb}) satisfies,
\begin{equation*}
    \ll \tp(t) \rl^2_{H^n_\pi} \leq  \frac{\(5\nu^nn!\)^2}{a}\EHn(0)e^{-at},\qd
    \ll \tm(t) \rl^2_{H^n_\pi}  \leq  \frac{\(5\nu^nn!\)^2}{b}\EHn(0)e^{-bt},
\end{equation*}
where $\EHn(0) = a\ll \tp(0) \rl^2_{H^n_\pi}+ b\ll \tm(0) \rl^2_{H^n_\pi}$. Here $C_0, \nu$ are constants independent of $n$, $C_0 =(5^22^5Ac^2)^{-1} $, $\nu= \k L$ and $L,A, \k$ are constants defined in (\ref{def of L}), (\ref{def of A}), (\ref{def of k}) respectively.  
\end{theorem}

\begin{remark}
    The above theorem tells us that as long as the initial random perturbation around the steady state is small enough, then the perturbation will exponentially decay with a rate of $e^{-at}, e^{-bt}$ for $\tp, \tm$ respectively. 
\end{remark}

\section{Proof of Theorem \ref{thm: sensitivity} (The sensitivity analysis around the steady state)}
\label{sec: sensitivity}

In this section, we are going to analyze how $\EHn = a\ll \tp \rl^2_{H^n_\pi} + b \ll \tm \rl^2_{H^n_\pi}$ evolves in time by studying $E^n$, 
\begin{align}
\label{def of En}
    &E^n = a\ll \Tpn_\o \rl^2+ b\ll \Tmn_\o \rl^2, 
\end{align}
where $\Tpn_\o, \Tmn_\o, \Ttn_\o$ are similarly defined as
\begin{align}
 &\Tpn_\o = \(\o^*_0\tp, \o^*_1\pt_z\tp, \cdots, \o^*_n\pt_z^n\tp\),
\end{align}
for weights $\o_i^*$ defined as,
\begin{equation}
\label{def of o}
    \o_i = \frac{(i+1)^2}{\k^ii!}, \qd \o^*_i = L^{n-i}\o_i.
\end{equation}
After taking the integration of the result for $E^n$ in the random space over $\pi(z)dz$, we can get the results for $E^n_\pi$,
\begin{align}
    &\ll \Tpn_\o \rl^2_\pi = \int \ll \Tpn_\o \rl^2\pi(z)dz, \qd \ll \Tmn_\o \rl^2_\pi =\int \ll \Tmn_\o \rl_{H^n_\pi}^2\pi(z)dz, \qd \Epi = a\ll \Tpn_\o \rl^2_\pi+b\ll \Tmn_\o \rl^2_\pi.\label{def of sens energy}
\end{align}
Using the relationship between $E^n_\pi$ and $\EHn(t) = a\ll \tp(t) \rl^2_{H^n_\pi}+ b\ll \tm(t) \rl^2_{H^n_\pi}$, we can get the exponential decay for $\EHn$.


The most important part in the proof is stated in the following Lemma \ref{lemma: sensitivity nonlinear}, which will be proved later. 
\begin{lemma}
\label{lemma: sensitivity nonlinear}
For $\rinf$ under Condition \ref{def of k}, and any vector function $\Tpn, \Tmn, \Ttn$, the following inequalities hold
\begin{equation}
\label{sens ineq_1}
    \sum_{l=0}^n\(\o^*_l\)^2\pt_z^l\(\tp \tm\)\pt_z^l\tp\leq \frac{2A}{\g L^n} \ll \Tpn_\o\rl^2\ll \Tmn_\o\rl^2 + \frac{2\g}{L^n}\ll \Tpn_\o \rl^2, \qd \forall \g>0;
\end{equation}
\begin{equation}
\label{sens ineq_2}
     -\sum_{l=0}^n\(\o^*_l\)^2\pt_z^l\(\rinf\tt\)\pt_z^l\tt \leq 0.
\end{equation}
\end{lemma}
\begin{proof}
See Section \ref{sec: proof of lemma sensitivity nonlinear}.
\end{proof}

If one multiplies $\o^*_l$ to the two equations in (\ref{eq: pt_zl perturb}) and adds the two equations together, then sums $l$ from $0$ to $n$, one has, 
\begin{equation}
\label{eq: pt_zl_1}
\begin{aligned}
    &\frac{1}{2}\pt_tE^n = -a^2\ll \Tpn_\o \rl^2 -b^2\ll \Tmn_\o \rl^2  -c\sum_{l=0}^n\(\o^*_l\)^2\pt_z^l\(\rinf\tt\)\pt_z^l\tt -c\sum_{l=0}^n\(\o^*_l\)^2\pt_z^l\(\tp \tm\)\pt_z^l\tt.
\end{aligned}
\end{equation}
Based on (\ref{sens ineq_2}) in Lemma \ref{lemma: sensitivity nonlinear}, one can omit the third term on the RHS of (\ref{eq: pt_zl_1}). Furthermore, by setting $\g = \frac{L^n}{8c}$ in (\ref{sens ineq_1}), one can bound the nonlinear terms by
\begin{equation}
\label{eq: pt_zl_2}
\begin{aligned}
    \frac{1}{2}\pt_tE^n \leq& -a^2\ll \Tpn_\o \rl^2 -b^2\ll \Tmn_\o \rl^2 + c^2\frac{16A}{ L^{2n}} \ll \Tpn_\o\rl^2\ll \Tmn_\o\rl^2 + \frac{a^2}{4}\ll \Tpn_\o \rl^2 + \frac{b^2}{4}\ll \Tmn_\o \rl^2\\
    = &-\(\frac{3a}{4} - \frac{8c^2A}{ aL^{2n}}\ll \Tmn_\o\rl^2\)a\ll \Tpn_\o \rl^2 -\(\frac{3b}{4} - \frac{8c^2A}{ bL^{2n}}\ll \Tpn_\o\rl^2\)b\ll \Tmn_\o \rl^2. 
\end{aligned}
\end{equation}
Since (\ref{eq: pt_zl_2}) is similar to (\ref{eq: n=0}) in the proof of Lemma \ref{lemma: est_0}, by the continuity arguement, one can conclude that if initially, 
\begin{equation}
\label{eq: initial Tpn}
\begin{aligned}
    &\frac{8c^2A}{aL^{2n}}\ll \Tmn_\o(0) \rl_\pi^2 \leq \frac{a}{4}, \qd \frac{8c^2A}{bL^{2n}}\ll \Tpn_\o(0) \rl_\pi^2 \leq \frac{b}{4}, 
\end{aligned}     
\end{equation}
$\ll \Tpn_\o(t) \rl_\pi^2, \ll \Tmn_\o(t) \rl_\pi^2$ decay as follows, 
\begin{equation}
\label{eq: decay Tpn}
    \ll \Tpn_\o(t) \rl_\pi^2 \leq  \frac{\Epi(0)}{a}e^{-at},\qd \ll \Tmn_\o(t) \rl_\pi^2 \leq  \frac{\Epi(0)}{b}e^{-bt}.
\end{equation}
Now, we need to transfer (\ref{eq: initial Tpn}) and  (\ref{eq: decay Tpn}) to the Sobolev norm we want to estimate in the random space $\(\ll \tp \rl^2_{H^n_\pi}, \ll \tm \rl^2_{H^n_\pi}\)$. 
Since
\begin{equation*}
    \frac{1}{n!} \leq \frac{(i+1)^2}{i!} \leq 5, \qd \text{for }0\leq i \leq n,
\end{equation*}
so one has 
\begin{equation*}
    \frac{1}{\k^nn!}\leq \o_i^* \leq 5L^n, \qd \text{for }0\leq i \leq n,
\end{equation*}
which implies that, 
\begin{equation*}
    \(\frac{1}{\k^nn!}\)^2\ll \tp \rl^2_{H^n_\pi}\leq \ll \Tpn_\o(t) \rl_\pi^2 \leq \(5L^n\)^2\ll \tp \rl^2_{H^n_\pi}, 
\end{equation*}
and similar relationship can be obtained for $\ll \Tmn_\o(t) \rl_\pi^2$ and $\ll \tm \rl^2_{H^n_\pi}$. 
Therefore, the initial requirement (\ref{eq: initial Tpn}) becomes, 
\begin{equation*}
\begin{aligned}
    &\frac{5^28c^2A}{a}\ll \tm(0) \rl^2_{H^n_\pi} \leq \frac{a}{4}, \qd \frac{5^28c^2A}{b}\ll \tp(0) \rl^2_{H^n_\pi}\leq \frac{b}{4}, 
\end{aligned}     
\end{equation*}
then $\(\ll \tp \rl^2_{H^n_\pi}, \ll \tm \rl^2_{H^n_\pi}\)$ will decay as follows, 
\begin{equation*}
    \ll \tp(t) \rl^2_{H^n_\pi} \leq  \(5\k^nn!L^n\)^2\frac{\EHn(0)}{a}e^{-at},\qd
    \ll \tm(t) \rl^2_{H^n_\pi}  \leq  \(5\k^nn!L^n\)^2\frac{\EHn(0)}{b}e^{-bt},
\end{equation*}
where $\EHn(0) = a\ll \tp(0) \rl^2_{H^n_\pi}+ b\ll \tm(0) \rl^2_{H^n_\pi}$ and this is obtained from (\ref{eq: decay Tpn}). The above two equations give the final results in Theorem \ref{thm: sensitivity}. 

\subsection{Proof of Lemma \ref{lemma: sensitivity nonlinear}}
\label{sec: proof of lemma sensitivity nonlinear}
The following is the proof of Lemma \ref{lemma: sensitivity nonlinear}.
\begin{proof}
Expanding $\pt_z^l(\tp\tm)$ gives,
\begin{equation}
\label{nonlinear sens_0}
\begin{aligned}
    &\sum_{l=0}^n\(\o^*_l\)^2\pt_z^l\(\tp \tm\)\pt_z^l\tp = \sum_{l=0}^n\sum_{i=0}^l\(\o^*_l\)^2\bn{l}{i} \pt_z^i\tp \pt_z^{l-i}\tm\pt_z^l\tp . 
\end{aligned}
\end{equation}
 
First notice that 
\begin{equation}
\label{nonlinear sens_1}
\begin{aligned}
    &\lv\(\o^*_l\)^2\bn{l}{i} \pt_z^i\tp \pt_z^{l-i}\tm\pt_z^l\tp  \rv
    = \lv\frac{(l+1)^2L^n}{L^{i}L^{l-i}\k^i\k^{l-i}l!}\frac{l!}{i!(l-i)!}\pt_z^i\tp \pt_z^{l-i}\tm\(\o^*_l\pt_z^l\tp \)\rv\\
    =& \lv\frac{(l+1)^2}{(i+1)^2(l-i+1)^2L^n} \(\o^*_i\pt_z^i\tp\)\(\o^*_{l-i}\pt_z^{l-i}\tm\) \(\o^*_l\pt_z^l\tp\) \rv\\
    \leq& \frac{2}{L^n}\(\frac{1}{(i+1)^2} + \frac{1}{(l-i+1)^2}\) \(\o^*_i\pt_z^i\tp\)\(\o^*_{l-i}\pt_z^{l-i}\tm\) \(\o^*_l\pt_z^l\tp\),
\end{aligned}
\end{equation}
where the second inequality is because of $$(l+1)^2 \leq \((i+1)+(l-i+1)\)^2 \leq 2(i+1)^2 + 2(l-i+1)^2.$$
If one sums up the first part of (\ref{nonlinear sens_1}) over $i,l$, one has, 
\begin{equation}
\label{nonlinear sens_2}
\begin{aligned}
    &\frac{2}{L^n}\sum_{l=0}^n\sum_{i=0}^l \frac{1}{(i+1)^2} \(\o^*_i\pt_z^i\tp\)\(\o^*_{l-i}\pt_z^{l-i}\tm\) \(\o^*_l\pt_z^l\tp\) \\
    \leq & \frac{1}{\g L^n}\sum_{l=0}^n \(\sum_{i=0}^l \frac{1}{(i+1)^2} \(\o^*_i\pt_z^i\tp\)\(\o^*_{l-i}\pt_z^{l-i}\tm\)\)^2 +  \frac{\g}{L^n}\sum_{l=0}^n\(\o^*_l\pt_z^l\tp\)^2\\
    \leq & \frac{1}{\g L^n}\sum_{l=0}^n \(\sum_{i=0}^l \frac{1}{(i+1)^2}\)\(\sum_{i=0}^l \frac{1}{(i+1)^2}  \(\o^*_i\pt_z^i\tp\)^2\(\o^*_{l-i}\pt_z^{l-i}\tm\)^2\) + \frac{\g}{L^n}\ll \Tpn_\o \rl^2\\
    \leq & \frac{A}{\g L^n}\sum_{i=0}^n\sum_{l=i}^n \frac{1}{(i+1)^2}  \(\o^*_i\pt_z^i\tp\)^2\(\o^*_{l-i}\pt_z^{l-i}\tm\)^2+ \frac{\g}{L^n}\ll \Tpn_\o \rl^2\\
    \leq & \frac{A}{\g L^n}\sum_{i=0}^n \frac{1}{(i+1)^2}  \(\o^*_i\pt_z^i\tp\)^2\sum_{l-i=0}^n\(\o^*_{l-i}\pt_z^{l-i}\tm\)^2+ \frac{\g}{L^n}\ll \Tpn_\o \rl^2\\
    \leq &\frac{A}{\g L^n}\ll \Tpn_\o\rl^2 \ll \Tmn_\o\rl^2 + \frac{\g}{L^n}\ll \Tpn_\o \rl^2.
\end{aligned}
\end{equation}
The first inequality is obtained by applying Young's inequality, and then applying Cauchy-Schwartz inequality gives the second one. Since $l-i$ and $i$ are symmetric, so the second part of (\ref{nonlinear sens_1}) can be similarly bounded. Therefore, summing (\ref{nonlinear sens_1}) over $i,l$ gives an upper bound for the RHS of (\ref{nonlinear sens_0}). This implies
\begin{equation*}
\begin{aligned}
    \sum_{l=0}^n\(\o^*_l\)^2\pt_z^l\(\tp \tm\)\tp_l 
    \leq &\frac{2A}{\g L^n}\ll \Tpn_\o\rl^2 \ll \Tmn_\o\rl^2 + \frac{2\g}{L^n}\ll \Tpn_\o \rl^2.
    \end{aligned}
\end{equation*}
For the second inequality (\ref{sens ineq_2}), one first separates it into two parts, 
\begin{equation}
\label{nonlinear sens_3}
\begin{aligned}
    & -\sum_{l=0}^n\(\o^*_l\)^2\pt_z^l\(\rinf\tt\)\pt_z^l\tt  = -\sum_{l=0}^n\(\o^*_l\)^2\rinf\pt_z^l\tt\pt_z^l\tt -\sum_{l=1}^n\sum_{i=1}^{l}\(\o^*_l\)^2\bn{l}{i}\pt_z^i\rinf\pt_z^{l-i}\tt\pt_z^l\tt \\
    = &- \rinf\ll \Ttn_\o \rl^2 - \frac{2}{L^n}\sum_{l=1}^n\sum_{i=1}^{l}\(\frac{1}{(i+1)^2} + \frac{1}{(l-i+1)^2}\)\(\o_i^*\pt_z^i\rinf\)\(\o^*_{l-i}\pt_z^{l-i}\tt\)\(\o^*_l\pt_z^l\tt\)\\
    \leq &- \rinf\ll \Ttn_\o \rl^2 + \frac{1}{\g L^n}\sum_{l=1}^n\(\sum_{i=1}^{l}\frac{1}{(i+1)^2} \(\o_i^*\rinf_i\)\(\o^*_{l-i}\tt_{l-i}\)\)^2\\
    &+   \frac{1}{\g L^n}\sum_{l=1}^n\(\sum_{i=1}^{l}\frac{1}{(l-i+1)^2} \(\o_i^*\rinf_i\)\(\o^*_{l-i}\tt_{l-i}\)\)^2+\frac{2\g}{L^n}\sum_{l=0}^n \ll \Ttn_\o \rl^2\\
    \leq &- \frac{\rinf}{2}\ll \Ttn_\o \rl^2 + \frac{4}{L^{2n}\rinf}\sum_{l=1}^n\(\sum_{i=1}^{l}\frac{1}{(i+1)^2}\(\o_i^*\rinf_i\)\(\o^*_{l-i}\tt_{l-i}\)\)^2 \\
    &+  \frac{4}{L^{2n}\rinf}\sum_{l=1}^n\(\sum_{i=1}^{l}\frac{1}{(l-i+1)^2}\(\o_i^*\rinf_i\)\(\o^*_{l-i}\tt_{l-i}\)\)^2,
\end{aligned}
\end{equation}
where the second equality is obtained by applying (\ref{nonlinear sens_1}), then applying Young's inequality gives the first inequality, and setting $\g = \frac{L^n\rinf}{4}$ gives the last inequality. The second and third terms in the last inequality are similar to the first term in the second line of (\ref{nonlinear sens_2}), so according to the fourth line in (\ref{nonlinear sens_2}),  (\ref{nonlinear sens_3}) can be further simplified to
\begin{equation}
\label{eqn: I}
\begin{aligned}
    & -\sum_{l=0}^n\(\o^*_l\)^2\pt_z^l\(\rinf\tt\)\tt_l  \\
    \leq& - \frac{\rinf}{2}\ll \Ttn_\o \rl^2 + \frac{4A}{L^{2n}\rinf}\sum_{l=1}^n \sum_{i=1}^{l}\(\frac{1}{(i+1)^2}+\frac{1}{(l-i+1)^2} \) \(\o_i^*\rinf_i\)^2\(\o^*_{l-i}\tt_{l-i}\)^2\\
    =&  - \frac{\rinf}{2}\ll \Ttn_\o \rl^2 + \frac{4A}{\rinf}\sum_{l=1}^n\sum_{i=1}^{l} \(\frac{L^{-2i}}{(i+1)^2} + \frac{L^{-2i}}{(l-i+1)^2}\) \(\o_i\rinf_i\)^2\(\o^*_{l-i}\tt_{l-i}\)^2\\
    \leq&  - \frac{\rinf}{2}\ll \Ttn_\o \rl^2 + \frac{4A\k^2}{\rinf}\sum_{l=1}^n\sum_{i=1}^{l} \(\frac{L^{-2i}}{(i+1)^2} + \frac{L^{-2i}}{(l-i+1)^2}\) \(\o^*_{l-i}\tt_{l-i}\)^2\\
    = &- \frac{\rinf}{2}\ll \Ttn_\o \rl^2 + \frac{4A\k^2}{\rinf}\sum_{i=1}^n \frac{L^{-2i}}{(i+1)^2}\sum_{l=i}^{n} \(\o^*_{l-i}\tt_{l-i}\)^2 +\frac{4A\k^2}{\rinf}\sum_{i=1}^nL^{-2i} \sum_{l=i}^{n} \frac{\(\o^*_{l-i}\tt_{l-i}\)^2}{(l-i+1)^2}
\end{aligned}
\end{equation}
where the first equality  comes from the definition of $\o^*_i$ in (\ref{def of o}), and the second inequality is by Assumption \ref{cond: rinf}, $$\sup_{z\in I_z} \(\o_{i}\pt_z^{i}\rinf\)^2 \leq \k^2.$$ Furthermore, since
\begin{equation*}
\begin{aligned}
    \sum_{i=1}^n \frac{L^{-2i}}{(i+1)^2} \leq \sum_{i=1}^nL^{-2i }  \leq \frac{1}{(L^2-1)} \leq \frac{r^2}{16A\k^2},
\end{aligned}
\end{equation*}
by the definition of $L$ in (\ref{def of L}). Inserting it back to (\ref{eqn: I}) gives
\begin{equation*}
\begin{aligned}
    & -\sum_{l=0}^n\(\o^*_l\)^2\pt_z^l\(\rinf\tt\)\tt_l
    \leq - \frac{\rinf}{2}\ll \Ttn_\o \rl^2 + \frac{r}{4}\ll \Ttn_\o \rl^2  + \frac{r}{4}\ll \Ttn_\o \rl^2 \leq 0,
\end{aligned}
\end{equation*}
which completes the proof for the second inequality (\ref{sens ineq_2}).
\end{proof}

\section{The gPC-SG method}
\label{sec: gPC method}
\subsection{The numerical method}
In this section, we will introduce a numerical method for model (\ref{eq: model}), which enjoys spectral accuracy in the random space. 

For random variable $z$ with probability density function $\pi(z)$, there exists a corresponding orthogonal polynomial basis $\{\Phi_i\}_{i=0}^\infty$ with respect to the measure $\pi(z)dz$, which is orthonormal to each other in the weighted $L^2_\pi$ inner product,
\begin{align}
    \int_{I_z} \Phi_i\Phi_j \pi(z)dz = \d_{ij},
    \label{eq: orthonormal}
\end{align} 
where $\d_{ij}$ is the Kronecker delta function. The $K$-th order subspace is therefore spanned by $\{\Phi_i\}_{i= 0}^K$. As a popular numerical method, the generalized Polynomial Chaos stochastic Galarkin (gPC-SG) method is to  find the approximate solution in  the truncated $K$-th order subspace. That is, define the approximation solution of the perturbative $\tp, \tm$ in the form of,
\begin{align}
    \hpk(t,x,z) = \sum_{i=0}^{K} \hp_i(t, x)\Phi_i(z) ,\qd \hmk(t,x,z) = \sum_{i=0}^{K}  \hm_i(t,x)\Phi_i(z),
\end{align}
then insert $\hpk$, $\hmk$ into (\ref{eq: stationary}) and do Galerkin projection, so the approximation solution $\hpk, \hmk$ satisfies,
\begin{align}
\label{eq: approx}
\begin{cases}
    &\la\pt_t\hpk, \Phi_j\ra_\pi = \la -\(a+ac\rinf\)\hpk - bc\rinf \hmk - c\hpk\hmk,\Phi_j\ra_\pi,\qd 0\leq j\leq K,\\
&\la \pt_t \hmk, \Phi_j\ra_\pi = \la- \(b +bc\rinf\) \hmk -ac\rinf\hpk -c\hpk\hmk, \Phi_j\ra_\pi, \qd 0\leq j\leq K.
\end{cases}
\end{align}
Equivalently, (\ref{eq: approx}) can be written as a system of the deterministic coefficients of $\hpk$, $\hmk$, i.e. the vector functions $\Hpk(t, x)  = \(\hp_0(t, x),\cdots, \hp_K(t, x)\)^\top$,  $\Hmk(t, x)  = \(\hm_0(t, x),\cdots, \hm_K(t, x)\)^\top$ satisfiy,
\begin{equation}
\label{eq: determ}
\l\{
\begin{aligned}
    \pt_t \Hpk = -a\Hpk - ac\Up\Hpk - bc\Up\Hmk-c\(\sum_{i,j}\hm_i S^l_{ij}\hp_j\)_{l=0}^K, \\
    \pt_t \Hmk = -b\Hmk - bc\Up\Hmk - ac\Up\Hpk-c\(\sum_{i,j}\hm_i S^l_{ij}\hp_j\)_{l=0}^K, \\
\end{aligned}
\r.
\end{equation}
with initial data,
\begin{equation*}
\begin{aligned}
    \hp_j(0) = \la\tp(0,z), \Phi_j \ra_\pi, \qd \hm_j(0)  =   \la\tm(0,z), \Phi_j \ra_\pi , \qd 0\leq j \leq K.
\end{aligned}
\end{equation*}
Here $S^l, \Up$ are symmetric matrices defined as
\begin{align}
    S^l_{ij} = \int_{I_z} \Phi_i\Phi_j\Phi_l \,\pi(z)dz, \qd \Up_{ij} =  \int_{I_z} \rinf\Phi_i\Phi_j \,\pi(z)dz.
    \label{def E}
\end{align}

\subsection{Main results and proof sketch}
\label{sec: gPC result}
We will prove that the approximate solution obtained by the gPC-SG from solving the deterministic system (\ref{eq: determ}) has spectral accuracy. We will decompose the error of the approximate solution into two parts, one is the projection error, another is the Galerkin error.  The first part is determined by the regularity of the solution $(\tp, \tm)$ in the random space, while the second part is determined by the stability of the Galerkin system (\ref{eq: determ}).

Define the projection of the analytic perturbative solution $(\tp, \tm)$ onto the subspace $\{\Phi_i\}_{i=0}^K$ as, 
\begin{align}
    \bpk := \(\int \tp\bPk d\pi(z)\)  \cdot\bPk, \qd\bmk := \(\int \tm \bPk d \pi(\bz)\)\cdot \bPk,
\end{align}
where $\bPk(z)= \(\Phi_i\)_{i=0}^{ K}$ is the vector function that contains all basis functions up to the $K$-th order. Then we can decompose the error of the approximation solution $\(\hpk, \hmk\)$   into two parts,
\begin{align}
&\tp - \hpk = \underbrace{(\tp - \bpk)}_{\ppk} + \underbrace{(\bpk - \hpk)}_{ \epk},\label{eq: error_p}\\
& \tm - \hmk =  \underbrace{(\tm - \bmk)}_{\pmk} + \underbrace{(\bmk -\hmk)}_{\emk}\label{eq: error_m},
\end{align}
where $\(\ppk, \pmk\)$ represents for the  projection error, $\(\epk, \emk\)$ are errors from the stochastic Galerkin.
Especially, we set $(\Bpk, \Bmk)$ to be the vector function defined as,
\begin{equation}
\label{def of gPC error}
\begin{aligned}
&\epk = \Bpk\cdot \bPk: =\( \int (\tp - \hpk)\bPk d\pi(z)\) \cdot  \bPk ,\\
&\emk = \Bmk\cdot \bPk: =\( \int (\tm - \hmk)\bPk d\pi(z)\) \cdot  \bPk .\\
\end{aligned}
\end{equation}
Because of the orthonality of the bases, it is easy to check that
\begin{equation*}
\ll \epk \rl^2_\pi= \ll \Bpk \rl^2,\qd\ll \emk \rl^2_\pi= \ll \Bmk \rl^2 .\\
\end{equation*}

From Theorem \ref{thm: sensitivity}, one can bound $(\ppk,\pmk)$ as in the following Corollary.
\begin{corollary}
\label{coro}
Under the same initial condition as in Theorem \ref{thm: sensitivity}, the projection error decays in time exponentially according to,
\begin{align}
\label{def of D}
&    \ll \ppk \rl^2_\pi = \ll \tp - \bpk \rl^2_\pi \leq \frac{D\(\nu^{n}n!\)^2\EHn(0)}{a(K+1)^{2n}}e^{-at}, \qd\ll \pmk \rl^2_\pi = \ll \tm - \bmk \rl^2_\pi \leq  \frac{D\(\nu^{n}n!\)^2\EHn(0)}{b(K+1)^{2n}}e^{-bt},
\end{align}
for some constant $D$ related to the measure $\pi(z)dz$.
\end{corollary}

\begin{proof}
 (\ref{def of D}) comes from the classical approximation theorem of orthogonal basis, one can refer to Theorem 2.1 in \cite{canuto1982approximation}. For $\tp \in H^n_z$, there exists a constant $D$, such that 
\begin{align}
    \ll \tp - \bpk \rl_\pi^2 \leq D\frac{\ll \tp \rl^2_{H^n_z}}{(K+1)^{2n}},
\end{align}
then applying the result of Theorem \ref{thm: sensitivity} directly gives (\ref{def of D}).
\end{proof}

Since by Corollary \ref{coro}, we already have estimates for the projection error $(\ppk, \pmk)$, so in order to study the convergence rate of the gPC-SG method, we only need to analyze the Galerkin error $(\epk, \emk)$.  Estimates for $(\epk, \emk)$ are based on the stability of the gPC-SG method, which is stated in Theorem \ref{thm: stability}. Similar to the analysis we did to get the estimates for $\ll \tp \rl_{H^n_\pi}, \ll \tm \rl_{H^n_\pi}$, if one directly does the energy estimates on $\ll \Hpk \rl_\pi$, $\ll \Hmk \rl_\pi$, one will end up with a strong assumption on the initial data for large $K$. In order to avoid that, we add a weight $\mu_i$ to $\hp_i, \hm_i$, then under Assumptions \ref{cond: Phi} and \ref{cond: stab rinf}, we can get a stability result that is sharp in $K$. 
\begin{condition}
\label{cond: Phi}
     There exists a positive integer $p$, such that the basis functions $\{\Phi_i(z)\}_{i\geq0}$ satisfy,
    \begin{align}
    \label{def of p}
        \ll \Phi_i(z) \rl_{L^\infty} \leq \ta_i = (i+1)^p, \qd\text{for all } i\geq 0.
    \end{align}
\end{condition}
\begin{remark}
    This assumption, first introduced in \cite{Jintwophase}, combined with the weight $\u_i$ defined in (\ref{def of u}) guarantees that the initial data do not depend on $K$. For example, the bases of normalized Legendre polynomials, which corresponds to uniform distribution in $[-1, 1]$, satisfy the above condition with $p = 1/2$; The bases of normalized Chebyshev polynomials, which corresponds to the random variable with pdf $\pi(z) = \frac{2}{\pi\sqrt{1-z^2}}$, satisfy this condition with $p=0$. 
\end{remark}

\begin{condition}
\label{cond: stab rinf}
Let  $\rinf_i = \la \rinf, \Phi_i\ra_\pi$, we assume
\begin{equation}
\label{eq: stab rinf}
\sum_{j \geq 1}((j+1)^q\rinf_j)^2 \leq \frac{\(\rinf_0\)^2}{2^{2q+3}A},
\end{equation}
where the constant $A$ is defined in (\ref{def of A}), $q = p+2$, with $p$ defined in (\ref{def of p}); $\rinf_0 = \int_{I_z} \rinf \pi(z)dz$ is the expectation of $\rinf$.
\end{condition}
\begin{remark}
One sufficient condition for $\rinf$ is, 
\begin{equation}
\lv \rinf_j \rv^2 \leq \frac{C}{(j+1)^{2q+2}}, \qd \text{for } \forall j\geq 1, \qd C = \frac{\(\rinf_0\)^2}{2^{2q+4}AC_S}.
\end{equation}
This implies that the variance of $\rinf$, which is equal to $\sum_{j\geq 1}\(\rinf_j\)^2$, has to be small enough. 
\end{remark}

We further define $\Hpk_\mu, \Hmk_\mu$ as weighted approximate solution 
\begin{align}
    \Hpk_\u = \(\u_0\hp_0, \cdots, \u_K\hp_K\), \qd \Hmk_\u = \(\u_0\hm_0, \cdots, \u_K\hm_K\)
\end{align}
where $\u_i$ are weights defined as,
\begin{equation}
\label{def of u}
    \u_i = (i+1)^q, \qd\text{for}\qd q = p+2,
\end{equation}
and here $p$ is the positive constant defined in (\ref{def of p}).

\begin{theorem}
\label{thm: stability}
(Stability of the gPC-SG method) Under Assumptions \ref{cond: Phi} and \ref{cond: stab rinf}, for the approximate perturbative solution $(\Hpk, \Hmk)$ obtained by the gPC-SG method, if initially 
\begin{equation}
\label{eq: stab_cond}
\begin{aligned}
    \ll \Hmk_\u(0) \rl^2 \leq a^2\h{C}_0, \qd \ll \Hpk_\u(0) \rl^2 \leq b^2\h{C}_0,
\end{aligned}
\end{equation}
then it decays in time as follows
\begin{equation}
\begin{aligned}
    \ll \Hpk_\u(t)\rl^2\leq\frac{1}{a}\Ek(0)e^{-at} ,\qd 
\ll \Hmk_\u(t)\rl^2\leq\frac{1}{b}\Ek(0) e^{-bt},
\end{aligned}
\end{equation}
where $\Ek =a\ll \Hpk_\u(t)\rl^2 + b\ll \Hmk_\u(t)\rl^2$. Here $\h{C}_0 = (2^{2q+6}c^2A)^{-1}$ and  $A$ is defined in (\ref{def of A}). 
\end{theorem}

The above theorem will be proved in Section \ref{sec: stability}. It tells us that the gPC-SG method is stable under some smallness assumption on the initial data. Based on the above result, we can prove the spectral accuracy of the gPC-SG method, which is stated in Theorem \ref{thm: spectral}. Before we state the theorem, we first introduce the Sobolev constant $C_S$,
\begin{equation}
    \label{def of C_S}
    \ll \tp\rl^2_{L^\infty_\bz} \leq C_S\ll \tp \rl^2_{H^1_\bz}, \qd \text{for }\forall \tp \in H^1_\bz.
\end{equation}

\begin{theorem}
\label{thm: spectral}
(Spectral accuracy of the gPC-SG method) Under Assumptions \ref{cond: rinf}, \ref{cond: Phi}, \ref{cond: stab rinf}, and in addition, initially the exact solution $(\tp, \tm)\in H^n_\pi$, and the approximate solution $(\hpk, \hmk)$ satisfies, 
\begin{equation*}
    \ll \tm(0) \rl_{H^n_\pi}^2 \leq a^2C_0, \qd \ll \tp(0) \rl_{H^n_\pi}^2 \leq b^2C_0,\qd \ll \Hmk_\u(0) \rl^2 \leq a^2\h{C}_0, \qd \ll \Hpk_\u(0) \rl^2 \leq b^2\h{C}_0,
\end{equation*}
then $(\hpk, \hmk)$ converges to $(\tp, \tm)$ according to,
\begin{equation*}
\begin{aligned}
    &\ll \tp - \hpk \rl_\pi^2 \leq \frac{C(n)}{a(K+1)^{2n}}e^{-at},\qd
    \ll \tm - \hmk \rl_\pi^2 \leq \frac{C(n)}{b(K+1)^{2n}}e^{-bt},
\end{aligned}
\end{equation*}
where $C_0 = \frac{b}{a+b}(5^22^6\nu^2c^2AC_S)^{-1}, \h{C}_0 = \frac{a}{a+b}\(2^{2q+6}c^2A\)^{-1}$, $C(n) = D\(1+I_0\)\nu^{2n}(n!)^2\EHn(0)$, $I_0 = \(32c^2R^2 +1 \)$, $A, R, D, C_S$ are constants defined in (\ref{def of A}), (\ref{def of r, R}),  (\ref{def of D}), (\ref{def of C_S}) respectively and $\nu$ is the same constant as in Theorem \ref{thm: sensitivity}.
\end{theorem}

\section{Proof of Theorem \ref{thm: stability} (Stability of the gPC-SG method)}
\label{sec: stability}
In this section, we will study the stability of the gPC-SG method for this model. We will use energy estimates to analyze $\Ek =a\ll \Hpk_\u(t)\rl^2 + b\ll \Hmk_\u(t)\rl^2 $. Similar to the proof in the sensitivity analysis in Section \ref{sec: sensitivity}, the most important part in the proof is how to bound the nonlinear term and the linear term with coefficient $\rinf(z)$ properly. We use the weight $\u_i$ to make the upper bound of this two terms independent of $K$, and it is stated in Lemma \ref{lemma: mSp}, which will be proved in Appendices \ref{proof: mSp}.   

By multiplying $a(\Hpk)^\top U^2$ and $b(\Hmk)^\top U^2$ with $U = \text{diag}(\u_0, \cdots, \u_K)$ to the two systems in (\ref{eq: determ}) respectively, one has,
\begin{equation}
\begin{aligned}
\frac{1}{2}\pt_t\Ek
\leq& - a^2\ll \Hpk_\u \rl - b^2\ll \Hmk_\u \rl -c\(\Htk\)^\top U^2 \Up \Htk  -c\sum_{l=0}^n \sum_{i,j} \u_l^2\hp_i S^l_{ij} \hm_j\hht_l,
\end{aligned}
\end{equation}
where $\hht = a\hp_l+b\hm_l$ and $\Htk = \(\hht_0, \cdots, \hht_K\)$. In the following Lemma \ref{lemma: mSp}, by (\ref{stab rinf}) one can omit the third term on the RHS of the above equation; by (\ref{stab nonlinear}), and setting $\gamma = \frac{1}{2c}$, one can bound the last term by,
\begin{equation}
\begin{aligned}
\frac{1}{2}\pt_t\Ek\leq & - a^2\ll \Hpk_\u \rl^2 - b^2\ll \Hmk_\u \rl^2
+ 2^{2q+5}c^2A \ll\Hmk_\u\rl \ll \Hpk_\u \rl^2  + \frac{a^2}{4}\ll \Hpk_\u \rl^2 + \frac{b^2}{4}\ll \Hmk_\u \rl^2\\
\leq & - \(\frac{3a}{4} -\frac{c^22^{2q+4}A}{a}\ll\Hmk_\u\rl^2  \)a\ll \Hpk_\u \rl_{H^1_x}^2 - \(\frac{3b}{4}  -\frac{c^22^{2q+4}A}{b}\ll\Hpk_\u\rl^2  \)b\ll \Hmk_\u \rl^2.
\end{aligned}
\end{equation}
Since the above inequality is similar to (\ref{eq: pt_zl_2}), by the continuity theorem, one gets similar result for $\ll \Hpk_\u\rl^2,\ll \Hmk_\u \rl$, which completes the proof.

\begin{lemma}
\label{lemma: mSp}
For $S^l$ defined in (\ref{def E}), the following inequality holds, 
\begin{equation}
\label{stab nonlinear}
    \sum_{l=0}^n \sum_{i,j} \u_l^2\hp_i S^l_{ij} \hm_j\hp_l \leq \frac{2^{2q+3}A}{\g}  \ll\Hmk_\u\rl \ll \Hpk_\u \rl^2 + \frac{\g}{2}\ll \Hpk_\u \rl^2,
\end{equation}
where  $q, A$  are constants defined in (\ref{def of u}), (\ref{def of A}).\\
For $\Up$ defined in (\ref{def E}), under Assumption \ref{cond: stab rinf}, the following inequality holds 
\begin{equation}
\label{stab rinf}
    -\(\Htk\)^\top U^2 \Up\Htk \leq 0.
\end{equation}
\end{lemma}

\begin{proof}
Similar proof of (\ref{stab nonlinear}) can be found in \cite{Zhu2017BE}, and based on (\ref{stab nonlinear}) and Assumption \ref{cond: stab rinf}, one can easily get (\ref{stab rinf}).  Therefore, we put the details of the proof in Appendix \ref{proof: mSp}. 
\end{proof}


\section{Proof of Theorem \ref{thm: spectral} (Spectral accuracy of the gPC-SG method)}
\label{sec: spectral}
In this section, we will prove the spectral accuracy of the gPC-SG method based on Theorems \ref{thm: sensitivity} and \ref{thm: stability}. We will use energy estimates to analyze $\Eek$,
\begin{equation}
    \label{def of Eek}
    \Eek =  a\ll \Bpk \rl^2 + b\ll \Bmk \rl^2.
\end{equation}
Project (\ref{eq: perturb}) onto the truncated subspace $\{ \bPk\}$, and then subtract the approximate perturbative system (\ref{eq: approx}) from it, one has the following system for $(\Bpk, \Bmk)$,
\begin{empheq}[left=\empheqlbrace]{align}
    \pt_t\Bpk = &-a\Bpk - ac\Up\Bpk - bc\Up\Bmk - c\int \(a\rinf\ppk + b\rinf \pmk\)\bPk d\pi(z) \nonumber\\
    &- c\int\(\tp\tm - \hpk\hmk\)\bPk d\pi(z) ,\label{eq: Bpk}\\
    \pt_t \Bmk =& -b\Bmk -bc\Up\Bmk -ac\Up\Bpk - c\int \( b\rinf\pmk +a\rinf\ppk\)\bPk d\pi(z) \nonumber\\
     &-c\int\(\tp\tm - \hpk\hmk\)\bPk d\pi(z) \label{eq: Bmk}.
\end{empheq}
When one does energy estimates to the above system, the most difficult part lies in how to bound the last nonlinear term. We analyze this term in Lemmas \ref{lemma: spectral ineq} and \ref{lemma: non_spectral}. For other linear terms, notice that $r\leq \rinf(z) \leq R$ for all $z\in I_z$, so by Theorem 3.1 in \cite{xiu2009efficient}, $\Up$ has the following properties. One can also refer to Appendices for the proof.
\begin{proposition}
\label{coro: pos def}
For the steady state $\rinf$ with lower bound and upper bound as in (\ref{def of r, R}), the constant matrix $\Up$ defined in (\ref{def E}) is a positive definite matrix and for any vector $\ba$, 
\begin{equation*}
    r\ll \ba \rl^2 \leq \ba^\top \Up \ba \leq R\ll \ba \rl^2.
\end{equation*}
\end{proposition}

Therefore, if one does dot product of $a\Bpk, b\Bmk$ to the two equations respectively, then add them together, after applying the above Proposition, one has
\begin{equation}
\label{spectral est_1}
\begin{aligned}
    &\pt_t\Eek \\
     \leq &-a^2\ll \Bpk \rl^2-b^2\ll \Bmk \rl^2 - c\la \rinf \(a\ppk + b\pmk\), a\epk + b\emk\ra_\pi\\
    &\underbrace{- c\la\tp\tm - \hpk\hmk,a\epk+b\emk\ra_\pi}_{IV} \\
    \leq &-a^2\ll \Bpk \rl^2-b^2\ll \Bmk \rl^2 + c^2\ll \rinf \rl^2_{L^\infty_z} \(8a^2\ll \ppk \rl^2_\pi + 8b^2\ll \pmk\rl^2_\pi\) \\
    &+ \frac{a^2}{8}\ll \Bpk \rl^2 + \frac{b^2}{8}\ll \Bmk\rl^2 + \(16c^2C_S\ll \tp \rl^2_{H^1_z} + \frac{b^2}{8}\)\ll\Bmk\rl^2+\(64c^2A\ll \Hmk_\u \rl^2+\frac{a^2}{8}\)\ll \Bpk \rl^2 \\
    & + 16c^2C_S\ll \tp \rl^2_{H^1_z}\ll\pmk\rl_\pi^2+ 64c^2A\ll \Hmk_\u \rl^2\ll \ppk \rl^2_\pi\\
    = &-\(\frac{3a}{4} - \frac{64c^2A}{a}\ll \Hmk_\u \rl^2\)a\ll \Bpk \rl^2-\(\frac{3b}{4} - 16c^2C_S\ll \tp \rl^2_{H^1_z}\) b\ll \Bmk \rl^2\\
    &+\underbrace{\(8a^2c^2R^2+ 64c^2A\ll \Hmk_\u \rl^2\)\ll \ppk \rl^2_\pi + \(8b^2c^2R^2+ 16c^2C_S\ll \tp \rl^2_{H^1_z}\)\ll \pmk\rl^2_\pi}_{J(t)},
\end{aligned}
\end{equation}
where Young's inequality and Lemma \ref{lemma: non_spectral} are applied to the second inequality.
Based on inequality (\ref{spectral est_1}), if 
\begin{equation}
\label{spect init cond}
     64c^2A\ll \Hmk_\u(t) \rl^2\leq\frac{a^2}{4}, \qd  16c^2C_S\ll \tp(t) \rl_{H^1_z}^2 \leq \frac{b^2}{4},
\end{equation}
then one has
\begin{equation}
\label{eqn: spectral_proc}
\begin{aligned}
    \frac{1}{2}\pt_t\Eek \leq & -\frac{a^2}{2} \ll \Bpk \rl^2- \frac{b^2}{2}\ll \Bmk \rl^2 + J(t).
\end{aligned}
\end{equation}
If $J(t)$ can be bounded for $\forall t>0$, then one can have exponential decay of $\Eek$. But first, let us check when assumption (\ref{spect init cond}) is satisfied. By Theorems \ref{thm: stability} and \ref{thm: sensitivity}, one has,
\begin{equation*}
    \ll \Hmk_\mu \rl^2 \leq \frac{\Ek(0)}{b} \leq a(a+b)\h{C}_0, \qd \ll \tp \rl^2_{H^1_\pi} \leq \frac{(5\nu)^2 E_{H^1_\pi}(0)}{a} \leq b(a+b)5^2\nu^2 C_0.
\end{equation*}
Therefore, as long as 
\begin{equation}
    \label{eqn: spect cond_1}
    \begin{aligned}
        &\h{C}_0 \leq \min\l\{\frac{a}{(a+b)}\frac{1}{2^8c^2A}, \frac{1}{2^{2q+6}c^2A}\r\} \leq \frac{a}{a+b}\(2^{2q+6}c^2A\)^{-1}, \\
        &C_0 \leq \min\l\{\frac{b}{(a+b)}\frac{1}{5^22^6\nu^2c^2C_S}, \frac{1}{5^22^5Ac^2}\r\} \leq \frac{b}{a+b}(5^22^6\nu^2c^2AC_S)^{-1}
    \end{aligned}
\end{equation}
(\ref{eqn: spectral_proc}) is satisfied, and then the error $\Eek$ satisfies (\ref{eqn: spectral_proc}). Integrating (\ref{eqn: spectral_proc}) over $t$ gives, 
\begin{equation*}
\begin{aligned}
    a\ll \Bpk(t) \rl^2+ b\ll \Bmk(t) \rl^2 \leq & 2\int_0^tJ(s)ds -a^2 \ll \Bpk \rl^2-b^2 \ll \Bmk \rl^2,
\end{aligned}
\end{equation*}
where $\Eek(0) = 0$ is used. Then separate it into two parts, one has
\begin{equation}
\label{spectral: ineq}
\begin{aligned}
    \ll \Bpk(t) \rl^2 \leq & \frac{2}{a}\int_0^tJ(s)ds -a \ll \Bpk \rl^2, \qd \ll \Bmk(t) \rl^2 \leq & \frac{2}{b}\int_0^tJ(s)ds -b \ll \Bmk \rl^2.
\end{aligned}
\end{equation}

Now we need to bound the term $\int_0^tJ(s)ds $. Insert (\ref{eqn: spectral_proc}) and Corollary \ref{coro} into $J(t)$, 
\begin{equation*}
    J(t) \leq \(8a^2c^2R^2 +  \frac{a^2}{4}\)\frac{D\(\nu^{n}n!\)^2\EHn(0)}{a(K+1)^{2n}}e^{-at} + \(8b^2c^2R^2 + \frac{b^2}{4}\)\frac{D\(\nu^{n}n!\)^2\EHn(0)}{b(K+1)^{2n}}e^{-bt},
\end{equation*}
which implies that
\begin{equation*}
\begin{aligned}
    &2\int_0^t J(s) ds \\
    \leq& \frac{2}{a}\(8a^2c^2R^2 +  \frac{a^2}{4}\)\frac{D\(\nu^{n}n!\)^2\EHn(0)}{a(K+1)^{2n}} + \frac{2}{b}\(8b^2c^2R^2 + \frac{b^2}{4}\)\frac{D\(\nu^{n}n!\)^2\EHn(0)}{b(K+1)^{2n}}\\
    \leq& 2\(16c^2R^2 +  \frac{1}{2}\)\frac{D\(\nu^{n}n!\)^2\EHn(0)}{(K+1)^{2n}} \leq\frac{I_0D\(\nu^{n}n!\)^2\EHn(0)}{2(K+1)^{2n}} ,
\end{aligned}
\end{equation*}
where $I_0 =32c^2R^2 + 1$, so (\ref{spectral: ineq}) becomes
\begin{equation}
\begin{aligned}
    \ll \Bpk(t) \rl^2 \leq & \frac{I_0D\(\nu^{n}n!\)^2\EHn(0)}{a(K+1)^{2n}} -a \ll \Bpk \rl^2, \qd \ll \Bmk(t) \rl^2 \leq &\frac{I_0D\(\nu^{n}n!\)^2\EHn(0)}{b(K+1)^{2n}}-b \ll \Bmk \rl^2.
\end{aligned}
\end{equation}
Applying Grownwall's inequality gives,
\begin{equation}
\label{eqn: Bpk Bmk}
     \ll \Bpk \rl^2 \leq \frac{I_0D\(\nu^{n}n!\)^2\EHn(0)}{a(K+1)^{2n}}e^{-at},\qd \ll \Bmk \rl^2 \leq \frac{I_0D\(\nu^{n}n!\)^2\EHn(0)}{b(K+1)^{2n}}e^{-bt}.
\end{equation}
Therefore, By (\ref{eq: error_p}), 
\begin{equation*}
\begin{aligned}
    \ll \tp - \hpk \rl^2_\pi \leq \ll \ppk \rl^2_\pi + \ll \epk \rl^2_\pi
\end{aligned}
\end{equation*}
and inserting (\ref{def of D}), (\ref{eqn: Bpk Bmk}) gives, 
\begin{equation*}
\begin{aligned}
    \ll \tp - \hpk \rl^2_\pi \leq  D(1+I_0)\frac{\nu^{2n}(n!)^2\EHn(0)}{a(K+1)^{2n}}e^{-at}.
\end{aligned}
\end{equation*}
Similar inequality can be obtained for $\ll \tm - \hmk \rl^2_\pi$, which completes the proof of Theorem \ref{thm: spectral}.

\begin{lemma}
    \label{lemma: spectral ineq}
    For any function $$\tm(z) = \sum_{i=0}^\infty \tm_i\Phi_i(z),\qd \tp(z) = \sum_{i=0}^\infty \tp_i\Phi_i(z),$$  where $\tm_i = \int \tm \Phi d\pi(z)$, $\tp_i = \int \tp \Phi d\pi(z)$, the following inequality holds,
    \begin{equation*}
        \sum_{l=0}^K\(\int \tp\tm \Phi_l d\pi(z) \)^2 \leq 4A\sum_{i\geq0}\((i+1)^q \tp_i \)^2 \sum_{i\geq0} \tm_i^2,
    \end{equation*}
    where $q = p+2$ and $p,A$ are constants defined in (\ref{def of p}), (\ref{def of A}).
\end{lemma}
\begin{proof}
First we define a function $\chi_{ijl}$ of non-negative integer $i,j,l$, 
\begin{equation}
    \label{def of chi}
    \chi_{ijl} = \l\{
    \begin{aligned}
        &1, \qd \text{if}\qd i+j \geq l, \text{or } i+l \geq j, \text{or } j+l \geq i\\
        &0, \qd \text{ortherwise.}
    \end{aligned}
    \r.
\end{equation}
Then we note that 
\begin{equation}
\label{ineq_proc}
\begin{aligned}
    &S^l_{ij} = \chi_{ijl}\int \Phi_i\Phi_j\Phi_l \pi(z)dz \leq\chi_{ijl} \,\ta_{\min\{i,j,l\}} \ll \Phi_i \rl_\pi\ll \Phi_l\rl_\pi \leq \ta_{\min\{i,j,l\}}\chi_{ijl},
\end{aligned}
\end{equation}
with $\eta_i$ defined in (\ref{def of p}). Therefore,
\begin{equation*}
\begin{aligned}
    &\sum_{l=0}^K\(\int \tp\tm \Phi_l d\pi(z) \)^2
    =\sum_{l=0}^{K} \(\int \(\sum_{i\geq0}^{\infty}\tp_i\Phi_i\)\(\sum_{j\geq0}\tm_j\Phi_j\)\Phi_l d\pi(z)\)^2   \\
    =&  \sum_{l=0}^{K}
    \(\sum_{i\geq0}\sum_{j\geq0}\tp_i\tm_jS^l_{ij}\)^2  
    \leq  \sum_{l=0}^{K}
    \(\sum_{i\geq0}\sum_{j\geq0}\lv(i+1)^q\tp_i\rv\lv\tm_j\rv\frac{\ta_i}{(i+1)^q}\chi_{ijl}\)^2\\
    =&\sum_{l=0}^{K}
    \(\sum_{i\geq0}\lv(i+1)^q\tp_i\rv\frac{1}{(i+1)^{2}}\sum_{j\geq0}\lv\tm_j\rv\chi_{ijl}\)^2\\
    \leq &\sum_{l=0}^{K}
    \(\sum_{i\geq0}\frac{1}{(i+1)^{2}}\)\sum_{i\geq0}\frac{\((i+1)^q\tp_i\)^2}{(i+1)^{2}}\(\sum_{j\geq0}\lv\tm_j\rv\chi_{ijl}\)^2\\
    \leq& A\sum_{l=0}^{K}
    \sum_{i\geq0}\frac{\((i+1)^q\tp_i\)^2}{(i+1)^{2}}\sum_{j\geq0}\tm_j^2\chi_{ijl}\sum_{j\geq0}\chi_{ijl} 
    \leq  A\sum_{l=0}^{K}
    \sum_{i\geq0}\frac{2i+1}{(i+1)^{2}}\((i+1)^q\tp_i\)^2\sum_{j\geq0}\tm_j^2\chi_{ijl}\\
    \leq&  A\sum_{i\geq0}\frac{(2i+1)^2}{(i+1)^{2}}\((i+1)^q\tp_i\)^2\sum_{j\geq0}\tm_j^2
    \leq 4A\sum_{i\geq0}\((i+1)^q \tp_i \)^2 \sum_{i\geq0} \tm_i^2.
\end{aligned}
\end{equation*}
In the above estimates, the first inequality is because of (\ref{ineq_proc}), then the Cauchy-Schwartz inequality is applied in the second and third inequalities. In the fourth inequality, one uses the property of $\chi_{ijl}$, since for fixed $i, l$, $\chi_{ijl}$ is nonzero only if $l-i \leq j \leq l+i$, which implies that $\sum_{j\geq0}\chi_{ijl} \leq (2i+1)$. Similar property is applied in the fifth inequality for $\sum_{l}\chi_{ijl}$. The last inequality comes from $(2i+1)^2 \leq 4(i+1)^2$. 
\end{proof}

\begin{lemma}
\label{lemma: non_spectral}
The following inequality holds 
\begin{equation*}
\begin{aligned}
     &\la \tp\tm - \hpk\hmk, a\epk+b\emk \ra \\
     \leq&  \(16c^2C_S\ll \tp \rl^2_{H^1_z} + \frac{b^2}{8}\)\ll\Bmk\rl^2+\(64c^2A\ll \Hmk_\u \rl^2+\frac{a^2}{8}\)\ll \Bpk \rl^2+ 16c^2C_S\ll \tp \rl^2_{H^1_z}\ll\pmk\rl_\pi^2 \\
    & + 64c^2A\ll \Hmk_\u \rl^2\ll \ppk \rl^2_\pi.
\end{aligned}
\end{equation*}
\end{lemma}
\begin{proof}
First notice that 
\begin{equation*}
\begin{aligned}
    &\tp\tm - \hpk\hmk  = \(\tp\tm - \tp\hmk\) + \(\tp\hmk - \hpk\hmk\)
    = \underbrace{\tp(\emk+\pmk)}_{\circled{1}} +\underbrace{\hmk\(\ppk + \epk\)}_{\circled{2}}.
\end{aligned}
\end{equation*} 
Apply Young's inequality to the first part, one has 
\begin{equation}
\label{spectral nonlinear_1}
\begin{aligned}
    &\lv -c\la \circled{1}, a\epk + b\emk \ra \rv \\
    \leq& 16c^2C_S\ll \tp \rl_{H^1_z}^2\ll\Bmk\rl^2+ 16c^2C_S\ll \tp \rl_{H^1_z}^2\ll\pmk\rl_\pi^2 + \frac{a^2}{16}\ll \Bpk \rl^2+ \frac{b^2}{16}\ll \Bmk \rl^2,
\end{aligned}
\end{equation}
where the constant $C_S$ comes from the Sobolev Embedding (\ref{def of C_S}).
For the second part $\circled{2}$, using Lemma \ref{lemma: spectral ineq}, one has, 
\begin{equation}
\label{spectral nonlinear_2}
\begin{aligned}
    &\lv -c\la \circled{2} ,  a\epk+b\emk\ra_\pi \rv = c\(\int \circled{2} \bPk \pi(z)dz\)\cdot\(a\Bpk+b\Bmk\) \\
    \leq &16c^2\(\int \hmk\ppk \bPk \pi(z)dz\)^2 +16c^2 \(\int \hmk\epk \bPk \pi(z)dz\)^2 + \frac{a^2}{16}\ll \Bpk \rl^2+\frac{b^2}{16}\ll \Bmk \rl^2\\
    \leq &64c^2AC_S\(\sum_{i\geq0}\((i+1)^q \hm_i \)^2\) \(\ll\ppk \rl_\pi^2+ \ll \Bpk \rl^2\) + \frac{a^2}{16}\ll \Bpk \rl^2+\frac{b^2}{16}\ll \Bmk \rl^2\\
    = &64c^2AC_S\ll \Hmk_\u \rl^2\(\ll \Bpk \rl^2 + \ll \ppk \rl^2_\pi\)+ \frac{a^2}{16}\ll \Bpk \rl^2+\frac{b^2}{16}\ll \Bmk \rl^2.
\end{aligned}
\end{equation}
Adding (\ref{spectral nonlinear_1}) and (\ref{spectral nonlinear_2}) together completes the proof.
\end{proof}

\section{Numerical examples}
\label{sec: numerics}
\subsection{Coefficient of variation}
We want to check how the { presence of $\mu$RNA influences the noise in the concentration of unbound mRNA. One common way to perform this comparison is to compute the coefficient of variation (CV) of the mRNA content,  i.e. the ratio of the standard deviation to the mean. Indeed, we expect the presence of $\mu$RNA to reduce the mean in the content of mRNA, simply because binding to $\mu$RNA reduces the amount of unbound mRNA. Since we deal with distributions on the positive real line, the reduction of the mean is also likely to reduce the variance. However, we wish to show that the variance reduction obtained by the presence of $\mu$RNA is actually bigger than the mere reduction which would be obtained as a consequence of a reduction of the mean. This is the reason of considering the CV. A reduction of the CV by the presence of $\mu$RNA shows a reduction of the variance which is larger than the corresponding reduction of the mean. Specifically, we compare the CV on $\pinf(z)$ obtained from system (\ref{eq: model}) which includes $\mu$RNA production with the CV on the steady state $\tilde \rho^\infty$ of the equation where binding with $\mu$RNA is ignored, namely}
\begin{equation}
\label{linear model}
\pt_t\p = S(z) -a\p.
\end{equation}
{ We let CV$_{L}$ be the CV of the steady state obtained from (\ref{linear model}), i.e. without $\mu$RNA, while CV$_{NL}$ is the CV of $\pinf$ obtained from (\ref{eq: model}), i.e. with $\mu$RNA ('L' and 'NL' stand for 'linear' and 'nonlinear' as \eqref{linear model} is linear while \eqref{eq: model} is nonlinear).}

Figure \ref{fig: CV_1} shows how { CV$_{L}$ - CV$_{NL}$} varies for different random sources. Here we set  $a = b = c = 1$, $S(z) = kz+d$, where $z$ follows the uniform distribution in $[-1/2,1/2]$. { The five lines correspond to the choices $d = 1/3, 1/2, 1, 2, 5$. The horizontal axis is the value of $k^2/2$ which is equal to the variance. }

\begin{figure}[htbp]
\includegraphics[width=1\textwidth]{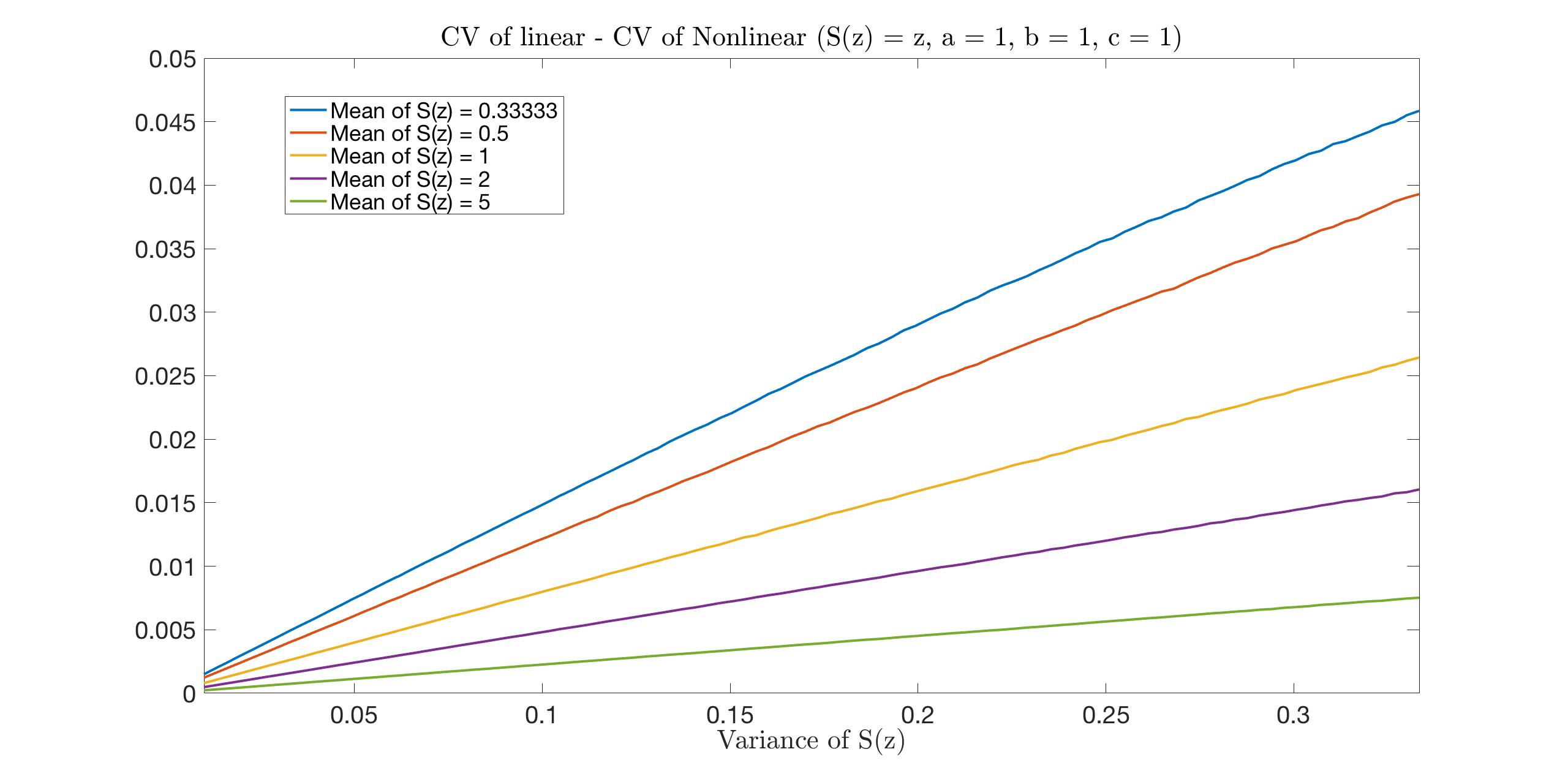}\\
\caption{$CV_{L} - CV_{NL}$ as a function of the variance of S(z) for different values of the mean of S(z) when $a = b = c = 1$, $S(z) = kz+d$, for $k\in [0,2]$ and $d = 1/3, 1/2,1,2,5$, where $z$ follows the uniform distribution in $[-1/2,1/2]$.}
\label{fig: CV_1}
\end{figure}

Figure \ref{fig: CV_2} { displays how CV$_{L}$ - CV$_{NL}$ depends on the parameters $a, b, c$.} Here we set $S(z) = 2z/3 + 1$, so the mean of the source is $1$ and the variance is $1/27$. 

\begin{figure}[htbp]
\includegraphics[width=1\textwidth]{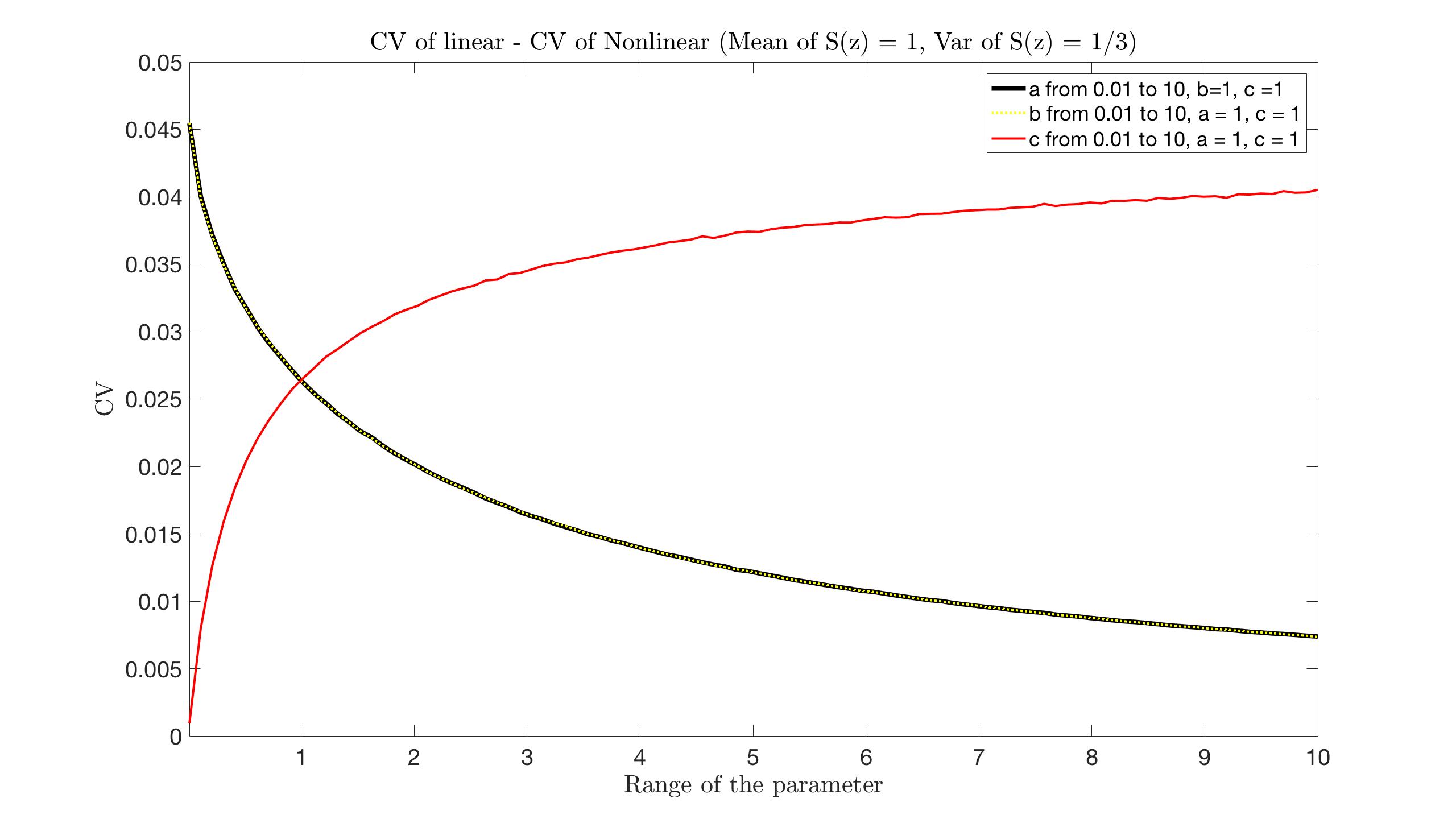}\\
\caption{$CV_{L} - CV_{NL}$ as a function of $a, b, c$, where we set $S(z) = 2z/3+1$, and $z$ follows the uniform distribution in $[-1/2,1/2]$.}
\label{fig: CV_2}
\end{figure}

From the two plots, one can see that the CV for the steady-state of the nonlinear system is always smaller than that of the linear system. Thus, the influence of $\mu$RNAs is always to decrease the uncertainty on the mRNA content. From Fig. \ref{fig: CV_1} we see that the influence of $\mu$RNAs increases as the intensity of the source decreases and its variance increases. From Fig. \ref{fig: CV_2} we deduce that the influence of $\mu$RNAs increases as their binding rate to mRNA $c$ increases. A larger binding rate means less unbound for mRNAs or $\mu$RNAs, which has a similar effect as a reduction of the source intensity. Indeed, the influence of $\mu$RNAs increases in both cases. Finally, From Fig. \ref{fig: CV_2}, an increase of either the degradation rate $a$ of mRNA or the degredation rate $b$ of $\mu$RNA both decrease the influence of $\mu$RNAs. In the latter case, this is understandable as the amount of unbound $\mu$RNA decreases and less noise reduction occurs. In the former one, this is less intuitive, as an increase of the degradation rate of mRNA should lead to a decrease of mRNA concentration relative to the $\mu$RNA concentration and should make the mRNA more sensitive to the presence of $\mu$RNA. This shows that nonintuitive outcome may occur from random perturbation of chemical kinetic systems. 

Finally, in spite of repeated attempts, we were not able to show the reduction of the CV in the presence of $\mu$RNA analytically. This may be the indication that for some randomness, this reduction does not  happen.

\section*{Acknowledgements}

PD acknowledges support by the Engineering and Physical Sciences 
Research Council (EPSRC) under grants no. EP/M006883/1 and EP/N014529/1, by the Royal 
Society and the Wolfson Foundation through a Royal Society Wolfson 
Research Merit Award no. WM130048 and by the National Science 
Foundation (NSF) under grant no. RNMS11-07444 (KI-Net). PD is on leave 
from CNRS, Institut de Math\'ematiques de Toulouse, France. 
PD thanks Matthias Merkenschlager from Faculty of Medicine, Institute of Clinical Sciences, Imperial College London, for bringing his attention on this problem. 

SJ acknowledges support from the Department of Mathematics of Imperial College London, where part of this research was conducted, through a Nelder fellowship and NSFC grants No. 11871297 and No. 31571071

YZ gratefully acknowledges the hospitality of the Department of Mathematics of Imperial College London, where part of this research was conducted.

\section*{Data availability}

No new data were collected in the course of this research.

\appendix
\section*{Appendices}
\addcontentsline{toc}{section}{Appendices}
\renewcommand{\thesubsection}{\Alph{subsection}}

\subsection{Proof of Proposition \ref{coro: pos def}}
\label{proof of energy est of t}
\setcounter{equation}{0}
\setcounter{theorem}{0}
\renewcommand\theequation{A.\arabic{equation}}
\renewcommand\thetheorem{A.\arabic{theorem}}
\begin{proof}
For any $K+1$-dimensional vector $a \neq 0$, 
\begin{equation}
\begin{aligned}
    a^\top Fa = &\sum_{i,j = 1}^{K} a_i F_{ij} a_j = \sum_{i,j = 1}^{K}  \int f(z)a_i\Phi_i(z)a_j\Phi_j(z)\pi(z)dz \\
    =&  \int f(z)\(\sum_{i=1}^Ka_i\Phi_i(z)\)\(\sum_{j = 1}^Ka_j\Phi_j(z)\)\pi(z)dz\\
    =& \int f(z)\(\sum_{i=1}^Ka_i\Phi_i(z)\)^2\pi(z)dz
    \geq C\sum_{i=1}^Ka_i^2,
\end{aligned}
\end{equation}
where the last inequality comes from the orthonormal relationship of $\Phi_i$ as shown in (\ref{eq: orthonormal}).
\end{proof}

\subsection{Proof of Lemma \ref{lemma: mSp}}
\label{proof: mSp}
\setcounter{equation}{0}
\setcounter{theorem}{0}
\renewcommand\theequation{B.\arabic{equation}}
\renewcommand\thetheorem{B.\arabic{theorem}}    
\begin{proof}
First note that for each $l$,
\begin{equation}
\label{eq: Sijl_twoparts}
\begin{aligned}
    &\(\u_l\sum_{i,j}\hm_i \hp_jS^l_{ij}\)\(\u_l\hp_l\) 
    \leq \frac{1}{2\g}\(\u_l\(\sum_{i\leq j } + \sum_{i> j}\)S^l_{ij} \lv\hm_i\hp_j\rv\)^2 + \frac{\g}{2}\(\u_l\hp_l\)^2\\
    \leq& \frac{\u_l^2}{\g} \(\sum_{i\leq j }\ta_i\chi_{ijl} \lv\hm_i\hp_j\rv\)^2 + \frac{\u_l^2}{\g}\(\sum_{i> j }\ta_j\chi_{ijl} \lv\hm_i\hp_j\rv\)^2 + \frac{\g}{2}\(\u_l\hp_l\)^2.
\end{aligned}
\end{equation}
The second inequality is because
\begin{equation}
\label{upper bound S_ij}
\begin{aligned}
    &S^l_{ij} = \chi_{ijl}\int \Phi_i\Phi_j\Phi_l \pi(z)dz \leq\chi_{ijl} \,\ta_{\min\{i,j,l\}} \ll \Phi_i \rl_\pi\ll \Phi_l\rl_\pi \leq \ta_{\min\{i,j,l\}}\chi_{ijl},
\end{aligned}
\end{equation}
where the first equality comes from the orthornality of $\Phi_i$, and $\ta_i$ defined in (\ref{def of p}) is the upper bound for $\Phi_i$. We estimate the first part of (\ref{upper bound S_ij}) as follows, 
\begin{equation*}
\begin{aligned}
     &\(\sum_{i\leq j }\ta_i\chi_{ijl} \lv\hm_i\hp_j\rv\)^2
    =\(\sum_{i\leq j }\frac{\chi_{ijl}}{(i+1)^2\u_j} \lv\u_i\hm_i\u_j\hp_j\rv\)^2 \\
    \leq & \(\sum_i\frac{1}{(i+1)^2}\) \sum_i \(\frac{\lv\mu_i\hm_i\rv}{i+1} \sum_{j\geq i}\frac{\chi_{ijl}\lv\u_j\hp_j\rv}{\u_j}\)^2\\
    \leq& A \sum_i \(\frac{\mu_i\hm_i}{i+1}\)^2 \(\sum_{j\geq i}\frac{\chi_{ijl}\lv\u_j\hp_j\rv}{\u_j}\)^2
    \leq A \sum_i \frac{\lv\mu_i\hm_i\rv^2}{(i+1)^2} \l[\sum_{j\geq i}\(\frac{1}{\u_j}\)^2\chi_{ijl} \sum_{j\geq i}\(\u_j\hp_j\)^2\chi_{ijl}\r].
\end{aligned}
\end{equation*}
The first equality is because of the definition of $\mu_i$ in (\ref{def of u}), then the Cauchy-Schwarz inequality is applied to the first and the last inequalities, while the second inequality comes from the definition of $A$ in (\ref{def of A}). Therefore, 
\begin{equation*}
\begin{aligned}
     &\sum_{l=0}^{K} \mu_l^2\(\sum_{i\leq j }\ta_i\chi_{ijl} \lv\hm_i\hp_j\rv\)^2
    \leq A \sum_{l=0}^{K} \sum_{i=0}^{K} \frac{\lv\mu_i\hm_i\rv^2}{(i+1)^2}  \l[\sum_{j\geq i}\(\frac{\u_l}{\u_j}\)^2\chi_{ijl} \sum_{j\geq i}\(\u_j\hp_j\)^2\chi_{ijl}\r]  \\
    \leq& 2^{2q}A \sum_{i=0}^{K} (2i+1)\frac{\lv\mu_i\hm_i\rv^2}{(i+1)^2}  \sum_{j\geq i}\(\u_j\hp_j\)^2\sum_{l=0}^{K} \chi_{ijl} \\
    \leq& 2^{2q}A \sum_{i=0}^{K} \frac{(2i+1)^2}{(i+1)^2}\(\mu_i\hm_i\)^2  \sum_{j\geq i}\(\u_j\hp_j\)^2
    \leq 2^{2q+2}A \ll\Hmk_\u\rl^2  \ll\Hpk_\u\rl^2.
\end{aligned}
\end{equation*}
Since $\chi_{ijl}$ is nonzero only if when $l \leq i+j$, and for $j\geq i$, this means $l \leq 2j$, so $\(\frac{\u_l}{\u_j}\)^2 \leq \frac{(2j+1)^{2q}}{(j+1)^{2q}} \leq 2^{2q}$. Furthermore, for fixed $i, l$, $\chi_{ijl}$ is nonzero only when $l-i\leq j \leq l+i$, this means the number of nonzero $\chi_{ijl}$ is $(2i+1)$. Therefore, \begin{equation*}
\sum_{j\geq i}\(\frac{\u_l}{\u_j}\)^2\chi_{ijl} \leq \frac{(2j+1)^q}{(j+1)^q} \leq 2^{2q}(2i+1),
\end{equation*} which gives the second inequality.  Similarly, one can obtain the third inequality. The fourth inequality is because of $(2i+1)^2\leq 2^2(i+1)^2$. \\
Since $i,j$ are symmetric, so the second part of (\ref{eq: Sijl_twoparts}) should have the same bound, hence, 
\begin{equation}
\label{stab nl_2}
    \sum_{l=0}^n\la \u_l\sum_{i,j} \hp_i \hm_jS^l_{ij}, \,\u_l\hp_l \ra \leq \frac{2^{2q+3}A}{\g} \ll\Hmk_\u\rl \ll \Hpk_\u \rl^2 + \frac{\g}{2}\ll \Hpk_\u \rl^2.
\end{equation}

For the second inequality (\ref{stab rinf}), first notice that, 
\begin{equation*}
\begin{aligned}
    &-\sum_{l=0}^K\la \rinf \htk \Phi_l, \,\u_l^2\hht_l \ra_\pi = -\sum_{l=0}^K\la \(\rinf_0 + \sum_{j\geq 1}\rinf_j\Phi_j\) \(\sum_{i=0}^K\hht_i\Phi_i\)\Phi_l, \,\u_l^2\hht_l \ra_\pi \\
    =&-\la\rinf_0 ,   \sum_{l=0}^K\sum_{i=0}^K\u_l^2\hht_i\hht_l\Phi_i\Phi_l \ra_\pi - \sum_{l=0}^K\sum_{j\geq1}\sum_{i=0}^K\u_l^2 \rinf_j \hht_i S^l_{ij}\hht_l \\
    =& -\rinf_0 \ll \Htk_\u\rl^2  - \sum_{l=0}^K\sum_{j\geq1}\sum_{i=0}^K\u_l^2 \rinf_j \hht_i S^l_{ij}\hht_l
\end{aligned}
\end{equation*}
where $S^l_{ij}$ is defined in (\ref{def E}). The third equality is because of the orthogonality of $\{\Phi_i\}_{i\geq0}$. For the last term, using the same technique one uses to get (\ref{stab nonlinear}), then one has
\begin{equation*}
\begin{aligned}
    &-\sum_{l=0}^K\la \rinf \htk \Phi_l, \,\u_l^2\hht_l \ra_\pi 
    \leq -\rinf_0 \ll \Htk_\u\rl^2 + \frac{2^{2q+3}A}{\g}\(\sum_{j\geq1}\(\u_j\rinf_j\)^2\)\ll \Htk_\u \rl^2 + \frac{\g}{2}\ll \Htk_\u \rl^2
\end{aligned}
\end{equation*}
Then set $\g = \rinf_0$, and by the condition on $\sum_{j \geq 1}(\u_j\rinf_j)^2 \leq \frac{\(\rinf_0\)^2}{2^{2q+3}A}$, one completes the proof for the second inequality (\ref{stab rinf}).     
    \end{proof}

\bibliographystyle{plain}

\begin{thebibliography}{10}

\bibitem{babuska2004galerkin}
Ivo Babuska, Ra{\'u}l Tempone, and Georgios~E Zouraris.
\newblock Galerkin finite element approximations of stochastic elliptic partial
  differential equations.
\newblock {\em SIAM Journal on Numerical Analysis}, 42(2):800--825, 2004.

{
\bibitem{Bleris_MolSystBiol11}
Leonidas Bleris, Zhen Xie, David Glass, Asa Adadey, Eduardo Sontag and Yaakov Benenson.
\newblock Synthetic incoherent feedforward circuits show adaptation to the amount of their genetic template.
\newblock {\em Molecular Systems Biology} 7:519, 2011.

\bibitem{Blevins_PlosGenetics15}
Rory Blevins, Ludovica Bruno, Thomas Carroll, James Elliott, Antoine Marcais, Christina Loh, Arnulf Hertweck, Azra Krek, Nikolaus Rajewsky, Chang-Zheng Chen, Amanda G. Fisher and Matthias Merkenschlager.
\newblock microRNAs regulate cell-to-cell variability of endogenous target gene expression in developing mouse thymocytes.
\newblock {\em PLOS Genetics} 11(2):e1005020, 2015.

\bibitem{Bosia_BMCSystBiol12}
Carla Bosia, Matteo Osella, Mariama El Baroudi, Davide Cor\`a and Michele Caselle.
\newblock Gene autoregulation via intronic microRNAs and its functions.
\newblock {\em  BMC Systems Biology}  6:131, 2012. 
}


\bibitem{canuto1982approximation}
Claudio Canuto and Alfio Quarteroni.
\newblock Approximation results for orthogonal polynomials in sobolev spaces.
\newblock {\em Mathematics of Computation}, 38(157):67--86, 1982.

\bibitem{cohen2010convergence}
Albert Cohen, Ronald DeVore, and Christoph Schwab.
\newblock Convergence rates of best n-term galerkin approximations for a class
  of elliptic spdes.
\newblock {\em Foundations of Computational Mathematics}, 10(6):615--646, 2010.

\bibitem{cohen2011analytic}
Albert Cohen, Ronald DeVore, and Christoph Schwab.
\newblock Analytic regularity and polynomial approximation of parametric and
  stochastic elliptic pde's.
\newblock {\em Analysis and Applications}, 9(01):11--47, 2011.

{

\bibitem{Deg_Her_Mir}
Pierre Degond, Maxime Herda, Sepideh Mirrahimi. 
\newblock A Fokker-Planck approach to the study of robustness in gene expression.
\newblock submitted. 

\bibitem{Gillespie_JPhysChem77}
Daniel T. Gillespie.
\newblock Exact stochastic simulation of coupled chemical reactions.
\newblock {\em The Journal of Physical Chemistry} 81(25): 2340--2361; 1977.


\bibitem{Herranz_GenesDev10}
H\'ector Herranz and Stephen M. Cohen.
\newblock MicroRNAs and gene regulatory networks: managing the impact of noise in biological systems.
\newblock {\em Genes Dev.} 24:1339--1344, 2010.
} 

\bibitem{HuJin18}
Jingwei Hu and Shi Jin.
\newblock Uncertainty quantification for kinetic equations.
\newblock {\em Uncertainty Quantification for Kinetic and Hyperbolic
  Equations}, pages 193--229, 2017.

\bibitem{jin2016august}
Shi Jin, Jian-Guo Liu, and Zheng Ma.
\newblock Uniform spectral convergence of the stochastic galerkin method for
  the linear transport equations with random inputs in diffusive regime and a
  micro-macro decomposition based asymptotic preserving method.
\newblock {\em Research in Math. Sci.,(in honor of the 70th
  birthday of Bjorn Engquist)}, 4:15, 2017. DOI 10.1186/s40687-017-0105-1.

\bibitem{ZhuJin18}
Shi Jin and Yuhua Zhu.
\newblock Hypocoercivity and uniform regularity for the
  Vlasov-Poisson-Fokker-Planck system with uncertainty and multiple scales.
\newblock {\em SIAM J. Math. Anal.}, to appear.

\bibitem{li2016uniform}
Qin Li and Li~Wang.
\newblock Uniform regularity for linear kinetic equations with random input
  based on hypocoercivity.
\newblock {\em SIAM/ASA J. Uncertainty Quantification}, 5(1):1193--1219, 2017.

\bibitem{JinLiu18}
Liu Liu and Shi Jin.
\newblock Hypocoercivity based sensitivity analysis and spectral convergence of
  the stochastic galerkin approximation to collisional kinetic equations with
  multiple scales and random inputs.
\newblock {\em (SIAM) Multiscale Modeling and Simulation}, 16:1085--1114, 2018.

{
\bibitem{Osella_PlosCB11}
Matteo Osella, Carla Bosia, Davide Cor\`a, Michele Caselle, 
\newblock The role of incoherent microRNA-mediated feedforward loops in noise buffering.
\newblock {\em PLoS Computational Biology}, 7: e1001101, 2011.
}

\bibitem{Jintwophase}
Ruiwen Shu and Shi Jin.
\newblock Uniform regularity in the random space and spectral accuracy of the
  stochastic galerkin method for a kinetic-fluid two-phase flow model with
  random initial inputs in the light particle regime.
\newblock {\em Math. Model Num. Anal.}, to appear.

{
\bibitem{vanKampen}
N. G. van Kampen.
\newblock Stochastic processes in physics and chemistry.
\newblock North Holland, 1981.
}


\bibitem{xiu2009efficient}
Dongbin Xiu and Jie Shen.
\newblock Efficient stochastic galerkin methods for random diffusion equations.
\newblock {\em Journal of Computational Physics}, 228(2):266--281, 2009.

\bibitem{ZhuVPFP18}
Yuhua Zhu.
\newblock A local sensitivity and regularity analysis for the
  Vlasov-Poisson-Fokker-Planck system with multi-dimensional uncertainty and
  the spectral convergence of the stochastic galerkin method.
\newblock Preprint.

\bibitem{Zhu2017BE}
Yuhua Zhu.
\newblock Sensitivity analysis and uniform regularity for the Boltzmann
  equation with uncertainty and its stochastic galerkin approximation.
\newblock Preprint.

\end{thebibliography}

\end{document}